\documentclass[11pt]{amsart}
\usepackage{amssymb,amsmath}
\usepackage{caption}

\usepackage{epsfig}
\usepackage{color}

\def\Bbb{\mathbb}
\def\Cal{\mathcal}

\def\eb{\varepsilon}

\def\R {\mathbb{R}}

\def\<{\left<}
\def\>{\right>}

\def\({\left(}
\def\){\right)}

\def\R{\Bbb R}

\def\eb{\varepsilon}
 \makeatletter
 \expandafter\def\csname r@tocindent4\endcsname{0pt}
 \makeatother

\newtheorem{proposition}{Proposition}[section]
\newtheorem{theorem}[proposition]{Theorem}
\newtheorem{corollary}[proposition]{Corollary}
\newtheorem{lemma}[proposition]{Lemma}

\theoremstyle{definition}

\newtheorem{remark}[proposition]{Remark}

\numberwithin{equation}{section}

\def \no#1#2#3 {{\bf #1} (#3), #2.}
\def \eds#1#2#3 {#1, #2, #3.}
\begin{document}
\title[Green's function asymptotics and interpolation inequalities  ] {Green's function asymptotics and sharp point-wise interpolation inequalities}

\author[A.A.Ilyin, S.V.Zelik] {Alexei A. Ilyin and Sergey V. Zelik}

\thanks{}
\address{\noindent Keldysh Institute of Applied Mathematics,
Miusskaya sq. 4, 125047 Moscow, Russia and \newline
Institute for Information Transmission Problems,  Bolshoy Karetny 19, 127994 Moscow, Russia;\newline
University of Surrey, Department of Mathematics,
Guildford, GU2 7XH, United Kingdom.}
\email{ilyin@keldysh.ru; s.zelik@surrey.ac.uk}

\medskip

\begin{abstract}
  We propose a general method for  finding sharp constants
in the imbeddings of the  Sobolev spaces  $H^m(\mathcal{M})$,
defined on a $n$-dimensional Riemann  manifold  $\mathcal{M}$ into
the space of bounded continuous functions, where $m>n/2$. The
method is based on the analysis of the asymptotics with respect to
the spectral parameter of the Green's function of the elliptic
operator of order 2m, the domain of the square root of which
defines the norm of the corresponding  Sobolev space. The cases of
the $n$-dimensional torus $\mathbb{T}^n$ and $n$-dimensional sphere
$\mathbb{S}^n$ are treated in detail, as well as some  manifolds
with boundary.  In certain cases when $\mathcal{M}$ is compact,
multiplicative inequalities with remainder terms of various types
are obtained. Inequalities with correction term for periodic functions
imply an improvement for the well-known Carlson inequalities.

\medskip\medskip\medskip\medskip\medskip

\textit{\large This paper is dedicated to the memory of Professor M.I.\,Vishik}

\end{abstract}

 \maketitle
\tableofcontents
\section{Introduction}\label{s0}

In this paper we propose a general method for finding sharp constants
in multiplicative inequalities of Sobolev--Gagliardo-Nirenberg type
characterizing the imbedding of the Hilbert Sobolev space $H^m(\mathcal{M})$
into the space of bounded continuous functions $C(\mathcal{M})$.  Here
$\mathcal{M}$ is an $n$-dimensional manifold and $m>n/2$. The inequalities
are as follows:
\begin{equation}\label{Int1}
\|u\|_{L_\infty(\mathcal{M})}^2\le
K\|u\|_{H^l(\mathcal{M})}^{2\theta}\|u\|_{H^m(\mathcal{M})}^{2(1-\theta)},
\end{equation}
where $-\infty<l<n/2<m<\infty$, so that $0<\theta=\frac{2m-n}{2(m-l)}<1$.

Of course, multiplicative inequalities are known to hold in a much more
general $L_p$ setting (see, for instance, \cite{Besov},\cite{Tri}), when the three norms in~\eqref{Int1} are the
$L_p$, $L_q$ and $L_r$-norms, $1\le p,q,r\le\infty$, and $\theta$ is
accordingly defined by  scale invariance; in this sense
we shall be dealing only with the $L_\infty-L_2-L_2$ case in this paper.
In the one-dimensional case, when $\mathcal{M}=\mathbb{R}$, $\mathbb{R}_+$,
or $(0,L)$ the corresponding interpolation inequalities are called
the inequalities for derivatives. There exists a vast literature devoted
to them, see, for instance, \cite{Ar},\cite{M-T} and the references therein.
On the whole line, the general $L_p-L_q-L_r$ case of a function and
its first-order derivative was completely settled in \cite{Nagy}.

The sharp constant in inequality~\eqref{Int1} on $\mathbb{R}$ was
found in~\cite{Taikov}, the more complicated case of the half-line
 $\mathbb{R}_+$ was solved in \cite{Gab}, the value of the constant
 in closed form was obtained in~\cite{Kal}.

The sharp constant in inequality~\eqref{Int1} for periodic functions with zero mean
was found in~\cite{I98JLMS}. In particular, for $l=0$
it was shown there  that
$$
\|u\|_\infty^2\le K(m)\|u\|^{2\theta}\|u^{(m)}\|^{2(1-\theta)},
\qquad u\in H^m_{\mathrm{per}}(0,2\pi),\ \int_0^{2\pi} u(x)dx=0,
$$
where $m>1/2$, $\theta=1-1/(2m)$ and
$$
K(m)=
\left(2m\theta^\theta(1-\theta)^{1-\theta}\sin\pi\theta\right)^{-1}.
$$
The constant (which is, in fact, the same as in the case
of the whole line) is sharp and no extremal functions exist.
We also observe that the first order inequality, namely,
\begin{equation}\label{Hardy-Carlson}
\|u\|_\infty^2\le \|u\|\|u'\|,
\qquad u\in H^1_{\mathrm{per}}(0,2\pi),\ \int_0^{2\pi} u(x)dx=0,
\end{equation}
was, in fact, proved much earlier in~\cite{Hardy}, as
a proof of the Carlson inequality. (In the end of this section we
discuss the connection of our results with the Carlson inequality in greater
detail.)

Sharp constants in inequalities~\eqref{Int1} on the sphere $\mathbb{S}^n$
were found in~\cite{I99JLMS}.

A comprehensive analysis of sharp constants in~\eqref{Int1} and in logarithmic
Brezis--Gallouet inequalities on the $n$-dimensional torus
$\mathbb{T}^n$ has been done in~\cite{Zelik}, where inequalities with correction terms were obtained for the first time. For example, the following inequalities hold
\begin{equation}\label{intro_tor.1.1}
\aligned
&\|u\|_{\infty}^2\le\|u\|\|u'\|-
\frac1{\pi}\|u\|^2,\quad &u\in H^1_{\mathrm{per}}(0,2\pi),\ \int_0^{2\pi} u(x)dx=0,\\
&\|u\|_{\infty}^2\le \frac{\sqrt2}{\sqrt[4]{27}}\|u\|^{3/2}\|u''\|^{1/2}-
\frac2{3\pi}\|u\|^2,\quad &u\in H^2_{\mathrm{per}}(0,2\pi),\ \int_0^{2\pi} u(x)dx=0,
\endaligned
\end{equation}
where all constants are sharp and no extremal functions exist. 
These inequalities have been verified in \cite{Zelik} with an
essential help of numerics. Purely analytic proof of them is given
in the present paper as one of the applications, see Theorem
\ref{Th.tor.1.1}.  

In the first part of this  paper we describe  a general method for finding sharp
constants in inequalities of the type~\eqref{Int1} including the inequalities
with correction terms. 
Namely, 
let $A$ and $B$ be  self-adjoint non-negative
elliptic differential operators on $\mathcal{M}$ of order $2m$ and
$2l$, respectively. (To describe the main ideas we may assume for a
moment in this introduction that $l=0$ and $B=Id$.) By the elliptic
regularity the kernel of $A$ is finite dimensional and consists of
smooth (orthonormal) functions:
$$
\ker A=\mathrm{span}\{\varphi_1,\dots,\varphi_k\}.
$$
We set
$\bar H^m:=H^m(\mathcal{M})\cap\ker A^\bot$ and
$\bar H^l:=H^l(\mathcal{M})\cap\ker A^\bot$. Then
the square roots of $A$ and $B$ define the equivalent norms on $\bar H^m(\mathcal{M})$ and
$\bar H^l(\mathcal{M})$:
$$
\|u\|_A^2:=(Au,u)\sim\|u\|_{H^m(\mathcal{M})}^2,\quad
\|u\|_B^2:=(Bu,u)\sim\|u\|_{H^l(\mathcal{M})}^2
$$
for $u\in \bar H^m$ and $u\in \bar H^l$, respectively.
\par
We fix an arbitrary point $\xi\in\mathcal{M}$ and consider the following analog of inequality \eqref{Int1}:
\begin{equation}\label{Int-our}
|u(\xi)|^2\le K\|u\|_B^{2\theta}\|u\|_A^{2(1-\theta)}.
\end{equation}
In particular, we are interested in the sharp constant $K=K(\xi)$
in this inequality. To study this problem, we introduce the
following maximization problem: for every  
 number $D\ge\lambda_0$, where $\lambda_0$ is the first eigenvalue of $B^{-1/2}AB^{-1/2}$ in
$\bar H:=L_2(\mathcal{M})\cap \ker A^\bot$, find
$\mathbb V(\xi,D)$, solving
\begin{equation}\label{Int2}
\mathbb V(\xi,D):=\sup\bigg\{|u(\xi)|^2:\ \ u\in\bar H^m,\ \ \|u\|_B^2=1,\ \ \|u\|_A^2=D\bigg\}.
\end{equation}
 Then, the inequality
$$
|u(\xi)|^2\le\|u\|_B^2\Bbb V\(\xi,\frac{\|u\|_A^2}{\|u\|_B^2}\)
$$
holds and, by definition, $\Bbb V$ is the smallest function for which it holds.
Thus, in particular, if the maximization problem \eqref{Int2} is solved, finding
the best constant $K$ in \eqref{Int-our} is reduced to finding the smallest  $K$
for which the following inequality holds:
$$
\Bbb V(\xi,D)\le K D^{1-\theta},\ \ D\ge\lambda_0.
$$
The solution of the maximization problem \eqref{Int2}  can be
expressed in terms of the Green's function of the following
elliptic operator of order $2m$:
$$
\Bbb A(\lambda):=A+\lambda B,\ \ \lambda>-\lambda_0.
$$
Namely, 
let $G_\lambda(x,\xi)$ be the Green's function of it in $\bar H$:
$$
\mathbb{A}(\lambda)G_\lambda(\cdot,\xi)=
\bar\delta(\cdot,\xi),\quad
\bar\delta(x,\xi):=\delta(x,\xi)-
\sum_{j=1}^k\varphi_j(x)\varphi_j(\xi),
$$
where $\delta(x,\xi)$ is the Dirac delta function.
From elliptic regularity we see that
 $G_\lambda(\cdot,\xi)\in \bar H^m\subset C(\mathcal{M})$.  
 Then, as shown 
 in Theorem~\ref{Th4.main}, there exist a unique
extremal function in~\eqref{Int2}
$$
u_{D,\xi}(x)=
\frac{G_{\lambda(D)}(x,\xi)}
{\|G_{\lambda(D)}(\cdot,\xi)\|_B},\quad\text{so that}\quad \mathbb V(\xi,D)=
\frac{G_{\lambda(D)}(\xi,\xi)^2}
{\|G_{\lambda(D)}(\cdot,\xi)\|_B^2}\,,
$$
where
 $\lambda(D)$
 is a  monotone increasing function 
with $\lim_{D\to\infty}\lambda(D)=\infty$,
$\lim_{D\to\lambda_0}\lambda(D)=-\lambda_0$, 
defined as a {\it unique} solution of 
$$
\|u_{D,\xi}\|_A^2=D.
$$
 Moreover, as shown in
Theorem~\ref{Th4.main1}, the sharp constant $K$ in the
multiplicative inequality \eqref{Int-our} can be expressed in terms
of the Green's function as follows: 
$$
K=K(\xi):=\frac{1}{\theta^\theta(1-\theta)^{1-\theta}}\cdot
\sup_{\lambda>0}\bigg\{\lambda^{\theta}G_\lambda(\xi,\xi)\bigg\}<\infty.
$$
In addition, the extremal function $u_*$ exists if and only if the
supremum with respect to $\lambda$ is attained at a finite point
$\lambda_*<\infty$ and then $u_*(x)=\mathrm{const}\,G_{\lambda_*}(x,\xi)$.
We note that the approach used in \cite{I98JLMS}, \cite{I99JLMS}
in the case of $\mathbb{S}^n,\ n\ge1$ eventually reduces to the same
one-dimensional maximization problem.

The further progress (inequalities with correction terms) is based on the knowledge of the asymptotic behavior of the Green's function as $\lambda\to\infty$:
\begin{equation}\label{Int4}
G_\lambda(\xi,\xi)=\lambda^{-1}\(g_1 \lambda^{1-\theta}+g_2+g_3\lambda^{\theta-1}+o(\lambda^{\theta-1})\),
\end{equation}
where $g_1=g_1(\xi)>0$, $g_2=g_2(\xi)$, $g_3=g_3(\xi)$ are some given
numbers. This expansion is assumed, and its verification is one of the main technical tasks in particular examples in the second part of the paper.

Given~\eqref{Int5}, we prove in Proposition~\ref{Prop6.as} that the solution of~\eqref{Int2} has the following asymptotics as  $D\to\infty$:
\begin{equation}\label{Int5}
\Bbb V(\xi,D)=g_1 S D^{1-\theta}+g_2\frac1\theta-
\frac12 S^{-1}\frac{g_2^2(1-\theta)-
2\theta g_1g_3}{\theta^3g_1} D^{\theta-1}+o(D^{\theta-1}) ,\ \
S:=\frac1{\theta^\theta(1-\theta)^{1-\theta}}.
\end{equation}
Suppose that the third term is negative
(this is always the case when $g_3\le0$), then
\begin{equation}\label{Int6}
\Bbb V(\xi,D)<g_1 SD^{1-\theta}+\frac{g_2}\theta,
\end{equation}
for large $D$. We shall see in certain examples in the second part of the paper that this inequality holds for all $D$. This implies the multiplicative inequality with correction term (which is negative if $g_2<0$):
$$
|u(\xi)|^2\le g_1 S\|u\|_B^\theta\|u\|_A^{1-\theta}+\frac{g_2}\theta\|u\|^2_B
$$
with best possible constants.

In the end of first part of the paper we give in Theorem~\ref{Th.var} the following variational characterization of $\Bbb V(\xi,D)$:
$$
\Bbb V(\xi,D)=\inf_{\lambda\in[-\lambda_0,\infty)}\{(\lambda+D)G_\lambda(\xi,\xi)\}.
$$

This formula is very useful for proving  in certain cases
(which are few!) that~\eqref{Int6} holds for all D.
The scheme is as follows.
 Usually it is impossible to find the unique minimizer $\lambda(D)$ explicitly, but we
always have the asymptotic formula for it
$$
\lambda(D)=\frac{\theta}{1-\theta}D+
\frac{g_2}{g_1}SD^\theta+\dots,
$$
see~\eqref{lambda(D)}. We somehow truncate this expansion and denote the result as $\lambda_*(D)$. Then to prove~\eqref{Int6} we proceed as follows
$$
\aligned
   \Bbb V(\xi,D)-g_1 SD^{1-\theta}-\frac{g_2}\theta
\le
(\lambda_*(D)+D)G_{\lambda_*(D)}(\xi,\xi)-g_1 SD^{1-\theta}-\frac{g_2}\theta.
\endaligned
$$
Now for a fixed $\xi$ the right-hand side is an explicit function
of $D$ only and more or less standard estimates can be used  prove
that it is negative for all $D$. 
In fact, inequalities
\eqref{intro_tor.1.1} as well
 as a number of other inequalities with lower order correctors
 mentioned below are verified using this scheme. 

In the  second part of the paper we consider examples and applications of the general approach
described in the first part and we  use the above scheme  in Theorem~\ref{Th.tor.1.1} for purely analytic proof of inequalities~\eqref{intro_tor.1.1} (the proof in \cite{Zelik} involves some reliable computer calculations).

We first deal with manifolds without boundary and consider the case $A=(-\Delta)^m$, $B=(-\Delta)^l$,
where $\Delta$ is the Laplace--Beltrami operator. Then $\bar H^s(\mathcal{M})=H^s(\mathcal{M})\cap
\{\int_\mathcal{M} u(x)d\mathcal{M}=0\}$.
 The asymptotic formula~\eqref{Int4} for the Green's function on the torus $\mathbb{T}^n$  is obtained by the Poisson summation formula. This works when $l=0$. For $l>0$ there is a singularity,
which can be removed by differentiation. This, in turn, produces the problem of finding integration constants. The corresponding technique was proposed in~\cite{Zelik} and it is further developed here.
We consider only one example $n=3$, $m=2$, $l=1$
and prove the following inequality on the three-dimensional torus $\mathbb{T}^3$:
$$
\|u\|_{\infty}^2\le \frac1{2\pi}\|\nabla u\|\|\Delta u\|-\frac{-\beta_3}{4\pi^3}\|\nabla u\|^2,\qquad
u\in\bar H^2(\mathbb{T}^3),
$$
where the integration constant $\beta_3$ is expressed in terms
of a super-exponentially convergent series, $\beta_3=-8.91363291758515127$. Both constants are sharp and no extremals exist.

Next we study  inequalities on spheres. On $\mathbb{S}^2$ we consider the case when $A=(-\Delta)^{m}$,
$m>1/2$, and $B=I$. The corresponding Green's function is independent of $\xi$ and is given by the series
$$
G(\lambda)=\frac1{4\pi}\mu^m\sum_{n=1}^\infty
(2n+1) \varphi(\mu n(n+1)),\quad\text{where}\quad \mu=\lambda^{-1/m},\quad \varphi(x)=\frac1{x^m+1}.
$$
Thus, we need to find the asymptotic behavior of
functions of the type
$$
F(\mu)=\sum_{n=1}^\infty
(2n+1) f(\mu n(n+1))\quad\text{as}\quad\mu\to 0,
$$
where $f$ is sufficiently smooth and sufficiently fast decays at infinity. This is achieved with the
help of the Euler--Maclaurin formula in Lemma~\ref{L:E-M}:
$$
F(\mu)=\frac1{\mu}\int_0^\infty f(x)dx-\frac23f(0)-
\frac1{15}\mu f'(0)+
O(\mu^2).
$$
This gives the asymptotic expansion of the type~\eqref{Int4} for the Green's function and, hence, the asymptotic expansion~\eqref{Int5} for
the solution $\mathbb{V}(D)$ of the corresponding
maximization problem, in which the second and the third terms turn out to be both
negative. Therefore a negative correction term may exist. For $m=2$ we show that this is indeed the case and the following inequality (with two sharp constants and no extremal functions) holds for
$u\in\bar H^2(\mathbb{S}^2)$:
$$
\|u\|_\infty^2\le
\frac1{4}
\|u\|\|\Delta u\|-\frac1{3\pi}\|u\|^2.
$$
For larger $m$ ($m=3$, $m=4$) the negative correction terms still exist, but are  smaller than
the second terms in the expansion for $\mathbb{V}(D)$.

On $\mathbb{S}^3$ we consider only one example
with $A=(-\Delta)^2$ and $B=-\Delta$. The series expressing the Green's function can be
summed in closed form, and along the same lines we obtain a sharp inequality for $u\in\bar H^2(\mathbb{S}^3)$:
$$
\|u\|^2_\infty\le\frac 1{2\pi}\|\nabla u\|
\|\Delta u\|-\frac3{4\pi^2}\|\nabla u\|^2.
$$

In the remaining part of the paper we consider manifolds with boundary.
We first prove a sharp multiplicative  inequality on the half-line for the Bessel operator~\cite{Lapt_unpubl}.

Then we consider the case when $u\in H^1_0(0,L)$. The correction term still exists, but is exponentially small, namely, the following inequality holds:
$$
\|u\|^2_\infty\le\|u\|\|u'\|\bigl(1-2e^{-\frac{L\|u'\|}{\|u\|}}\bigr).
$$
Both  coefficients on the  right-hand side
are sharp and no extremal functions exist.

For a second order inequality on the interval
$(0,L)$
$$
\|u\|^2_\infty\le K\|u\|_{L_2(0,L)}^{3/2}\|u''\|_{L_2(0,L)}^{1/2},
$$
in the case when $u\in H^2_0(0,L)$ we can use extension by zero, and therefore $K$ is the same as on $\mathbb{R}$, namely,
$K=\frac{\sqrt{2}}{\sqrt[4]{27}}$ (see, \cite{Taikov}, and also~\eqref{cR^n}). In going over to a wider space $u\in H_0^1(0,L)\cap H^2(0,L)$ the constant may  increase. In Theorem~\ref{Th:4th_order} we show that this is indeed the case:
$
K=\frac{\sqrt{2}}{\sqrt[4]{27}}\cdot \coth\frac\pi 2=
\frac{\sqrt{2}}{\sqrt[4]{27}}\cdot 1.09033\dots\,
$,
and, in addition, there exists a unique extremal function.

A somewhat opposite result is obtained in Theorem~\ref{Th:ketai}: the constant $\frac1{2\pi}$ in the inequality
$$
\|u\|_\infty^2\le\frac1{2\pi}\|\nabla u\|\|\Delta u\|,
$$
is sharp both for $u\in H_0^2(\Omega)$ and
$u\in H_0^1(\Omega)\cap H^2(\Omega)$,
$\Omega\subset\mathbb{R}^3$.

In conclusion we observe that inequalities
of the type~\eqref{Int1} have applications to various problems in partial differential equations
and mathematical physics. For example, the far-going generalization to the matrix-valued case~\cite{D-L-L} of
the simplest inequality
$$
\|u\|^2\le\|u\|\|u'\|, \qquad u\in H^1(\mathbb{R})
$$
gives the best-known estimates of the Lieb--Thirring constants for the negative trace of the Schr\"odinger operators in $\mathbb{R}^n$~\cite{L-T}.

Accordingly, inequalities with correction terms~\eqref{intro_tor.1.1} imply
by the method of~\cite{E-F}  a simultaneous bound for
the negative trace and the number of negative eigenvalues for the Schr\"odinger operators on $\mathbb{S}^1$~\cite{I12JST}.

Finally, inequalities~\eqref{intro_tor.1.1} provide an improvement to the  well-known Carlson inequality
\cite{Carl} (see also \cite{Carl_book} and the references therein):
for $a_k\ge0$
\begin{equation}\label{Carlson}
\left(\sum_{k=1}^\infty a_k\right)^2\le
\pi\left(\sum_{k=1}^\infty a_k^2\right)^{1/2}
\left(\sum_{k=1}^\infty k^2a_k^2\right)^{1/2},
\end{equation}
where the constant $\pi$ is sharp and the strict inequality holds unless all $a_k=0$. In fact, as shown in~\cite{Hardy}, inequalities~\eqref{Hardy-Carlson} and
\eqref{Carlson} are equivalent, hence an improvement to \eqref{Hardy-Carlson} results in an improvement to~\eqref{Carlson}.

Given a sequence $a_k\ge0$, we set $a_0=0$
and consider a real-valued function
$$u(x)=\sum_{k=-\infty}^\infty a_{|k|}e^{ikx}$$
with mean value zero, for which
$$
\|u\|_\infty=u(0)=2\sum_{k=1}^\infty a_k,\ \
\|u\|^2=4\pi\sum_{k=1}^\infty a_k^2,\ \
\|u'\|^2=4\pi\sum_{k=1}^\infty k^2a_k^2.\ \
$$
Now  inequalities~\eqref{intro_tor.1.1} can be reformulated as
follows
\begin{equation}\label{Carlson_impr}
\aligned
&\left(\sum_{k=1}^\infty a_k\right)^2\le
\pi\left(\sum_{k=1}^\infty a_k^2\right)^{1/2}
\left(\sum_{k=1}^\infty k^2a_k^2\right)^{1/2}-
\sum_{k=1}^\infty a_k^2,\\
&\left(\sum_{k=1}^\infty a_k\right)^2\le
\frac{\sqrt{2}\,\pi}{\sqrt[4]{27}}\left(\sum_{k=1}^\infty a_k^2\right)^{3/4}
\left(\sum_{k=1}^\infty k^4a_k^2\right)^{1/4}-
\frac23\sum_{k=1}^\infty a_k^2,
\endaligned
\end{equation}
where all constants are sharp and the inequalities are strict
unless all $a_k=0$. 
Moreover, the interpolation inequalities on the tori $\Bbb T^n$ discussed above can
 be naturally considered as multi-dimensional generalizations of the Carlson
 inequality. 

\medskip\medskip

In this paper we  use standard notation. Sometime the $C$-norm (or $L_\infty$-norm)  is denoted by $\|\cdot\|_\infty$
and the $L_2$-norm is denoted by $\|\cdot\|$.

\section{Part I. General theory}\label{s1}
\subsection{Assumptions and preliminaries}\label{ss1} Let $\mathcal M$
be an $n$-dimensional compact Riemann manifold and let
 $$
 (u,v):=\int_{\mathcal M}u(x)v(x)\operatorname{mes}(\,dx)
 $$
  be the standard scalar product in $H:=L^2(\mathcal M,\operatorname{mes})$ (where $\operatorname{mes}$ stands for the measure on $\mathcal M$ associated with the Riemann metric). As usual, we define the Sobolev spaces $W^{l,p}(\mathcal M)$, $1\le p\le\infty$ as spaces of distributions whose derivatives up to order $l$ belong to $L^p(\mathcal M)$ (this definition works for integer $l$ only, for non-integer and/or negative $l$, the spaces $W^{l,p}(\mathcal M)$ are defined in a standard way using the interpolation and duality methods, see e.g., \cite{Tri} for the details). For the case $p=2$, we will denote by $H^m=H^m(\mathcal M)$ the Sobolev space $W^{l,2}(\mathcal M)$
 \par
We assume that $A$ is an elliptic self-adjoint differential operator of order $2m$ on $\mathcal M$ with
smooth coefficients. This operator is supposed to be non-negative
\begin{equation}\label{1.non}
(Au,u)\ge0,\ \ u\in H^m,
\end{equation}
although it may have a non-zero kernel. Then, due to the classical elliptic
theory (see e.g., \cite{Tri}), the kernel is finite-dimensional and is
generated by smooth functions
\begin{equation}\label{1.kernel}
\operatorname{ker}A:=\operatorname{span}\{\varphi_1,\cdots,\varphi_k\},\ \ \varphi_i\in C^\infty(\mathcal M),\ \ A\varphi_i=0.
\end{equation}
Without loss of generality, we may assume that
$$
\|\varphi_i\|_H=1 \ \ {\rm and} \ \ (\varphi_i,\varphi_j)=0,\ \ i\ne j.
$$
We denote by $\bar H$ the orthogonal complement of
$\operatorname{ker} A$ in $H=L^2(\mathcal M)$ and define
\begin{equation}\label{1.barH}
\bar H^{s}:=H^s(\mathcal M)\cap \bar H,\ \ s\in\R.
\end{equation}
Then, due to the elliptic theory, $A$ is an isomorphism between $\bar H^{s+2m}$ and $\bar H^s$ for all $s\in\R$ and,
in particular, the equivalent norm in $\bar H^m$ is given by
\begin{equation}\label{1.Anorm}
\|u\|_A^2:=(Au,u),\ \ u\in \bar H^m.
\end{equation}
We also introduce the second elliptic non-negative and self-adjoint differential
operator $B$ of order $2l<2m$ on $\mathcal M$ with smooth coefficients such that
\begin{equation}\label{1.Bdef}
\operatorname{ker} B\subset\operatorname{ker}A,\ \ {\rm and}\ \ B\,\operatorname{ker} A\subset\operatorname{ker}A.
\end{equation}
Then, as not difficult to see, the operator $B$ is an isomorphism between $\bar H^{s+2l}$ and $\bar H^{s}$ for all $s\in\R$ and,
in particular, the equivalent norm in $\bar H^l$ is given by
\begin{equation}\label{1.Bnorm}
\|u\|_B^2:=(Bu,u),\ \ u\in\bar H^l.
\end{equation}
In addition, we have the analogue of the Poincare inequality:
\begin{equation}\label{Poincare}
\|u\|_{B}^2\le \lambda_0^{-1}\|u\|_A^2,\ \ u\in \bar H^m,
\end{equation}
where $\lambda_0>0$ is the minimal eigenvalue of $B^{-1/2}AB^{-1/2}$ in
$\bar H$.

\subsection{Interpolation inequality and associated variational problem}\label{ss2}
If
\begin{equation}\label{2.dim}
l<\frac n2<m,
\end{equation}
then for every $u\in\bar H^m$ the following interpolation inequality holds:
\begin{equation}\label{2.int}
\|u\|_{C(\mathcal M)}\le C\|u\|_{B}^\theta\|u\|_{A}^{1-\theta},\ \
 \frac n2=l\theta+m(1-\theta),
\end{equation}
see \cite{Besov}, \cite{Tri}. Our aim is to refine inequality
\eqref{2.int} and, in particular, to find the best constant $C=C(n,m,l,\mathcal M)$
 or/and the lower order extra terms in it, etc. To this end, we fix an arbitrary point
  $\xi\in\mathcal M$ and a positive $D$ and
 consider the following maximization problem:
\begin{equation}\label{2.max}
\mathbb V(\xi,D):=\sup\bigg\{|u(\xi)|^2:\ \ u\in\bar H^m,\ \ \|u\|_B^2=1,\ \ \|u\|_A^2=D\bigg\}.
\end{equation}
Indeed, in view of \eqref{2.int} and \eqref{Poincare} the function $\mathbb V$ is
well-defined for all $\xi\in\mathcal M$ and all $D\ge\lambda_0$.
On the other hand, due to the homogeneity,
\begin{equation}\label{2.Vint}
|u(\xi)|^2\le \|u\|^2_B \mathbb V\(\xi,\frac{\|u\|_A^2}{\|u\|_B^2}\),\ \ \xi\in\mathcal M,\ \ u\in\bar H^m.
\end{equation}
In particular, the best constant $C$ in \eqref{2.int} is the minimal one for
which the inequality
\begin{equation}\label{2.best}
\sup_{\xi\in\mathbb M}\mathbb V(\xi,D)\le C^2D^{1-\theta}
\end{equation}
holds for all $D\ge\lambda_0$. This reduces the study of
inequality \eqref{2.int} to the investigation of the maximization
problem \eqref{2.max}.

\subsection{Green's functions and reproducing functionals}\label{ss3}
In this subsection we prepare some technical tools which are necessary
to give the analytic description of the function $\Bbb V$ in terms of
the Green's functions of the appropriate elliptic operators on $\mathcal{M}$.
Namely, let
\begin{equation}\label{3.lambda}
\Bbb A(\lambda):=A+\lambda B,\ \ \lambda>-\lambda_0,
\end{equation}
and let the function $G_{\lambda}(x,\xi)$ solve
\begin{equation}\label{3.green}
\Bbb A(\lambda)G_\lambda(\cdot,\xi)=\bar{\delta}(\cdot,\xi),\ \
\bar\delta(x,\xi)=\delta(x,\xi)-\sum_{i=1}^k\varphi_i(\xi)\varphi_i(x),
\end{equation}
where $\delta(x,\xi)$ is the Dirac $\delta$-function at point
$\xi\in\mathcal{M}$ and $\bar\delta$ is its `projection' on the space
$\bar H$. Then, due to the Sobolev embedding theorem and \eqref{2.dim},
$$
\bar\delta(\cdot,\xi)\in\bar H^{-m}
$$
and, therefore, due to the elliptic regularity,
\begin{equation}\label{2.G}
G_\lambda(\cdot,\xi)\in \bar H^m\subset C(\mathcal M).
\end{equation}
In addition, since $\Bbb A(\lambda)$ is self-adjoint,
$G_\lambda(x,\xi)=G_\lambda(\xi,x)$ and, consequently, \eqref{2.G}
implies that
$$
 G_{\lambda}\in C(\mathcal M\times\mathcal M).
$$
Moreover,
\begin{equation}\label{3.G-u}
u(\xi)=(\bar\delta(\cdot,\xi),u)=(\Bbb A(\lambda)G_\lambda(\cdot,\xi),u)=
(G_\lambda(\cdot,\xi),\Bbb A(\lambda)u).
\end{equation}
for all $u\in\bar H^m$. In particular, taking $u(x):=G_\lambda(x,\xi)$,
we have
\begin{equation}\label{3.GG}
G_\lambda(\xi,\xi)=(\Bbb A(\lambda)G_\lambda(\cdot,\xi),G_\lambda(\cdot,\xi))=\|G_\lambda(\cdot,\xi)\|^2_A+\lambda\|G_\lambda(\cdot,\xi)\|^2_B>0.
\end{equation}
The following simple lemma is nevertheless the main technical tool for
the method of reproducing functionals and will allow us to find the
analytic expression for the function $\Bbb V$.
\begin{lemma}\label{Lem3.est} Let the above assumptions hold.
Then, for every $u\in\bar H^m$ and for every $\lambda>-\lambda_0$
\begin{equation}\label{3.main}
|u(\xi)|^2\le G_\lambda(\xi,\xi)\(\|u\|_A^2+\lambda\|u\|^2_B\).
\end{equation}
Moreover, the equality holds if and only if
$u(x)=c\, G_\lambda(x,\xi)$ for some $c\in\mathbb{R}$.
\end{lemma}
\begin{proof} Indeed, since $\Bbb A(\lambda)$ is positive definite, by the
Cauchy-Schwartz inequality,
$$
|u(\xi)|^2=(G_\lambda(\cdot,\xi),\Bbb A(\lambda)u)^2\le
(\Bbb A(\lambda)G_\lambda,G_\lambda)(\Bbb A(\lambda)u,u)=
G_\lambda(\xi,\xi)\(\|u\|^2_A+\lambda\|u\|^2_B\)
$$
and the equality here holds if and only if $u(x)=c G_\lambda(x,\xi)$.
Thus, the lemma is proved.
\end{proof}
We note that, up to the moment, $\lambda\ge-\lambda_0$ is a free parameter
in \eqref{3.main}. The next lemma
 shows that this parameter can be chosen in such way that the quotient
 $\|G_\lambda(\cdot)\|^2_A/\|G_\lambda(\cdot,\xi)\|^2_B$ achieves any
 prescribed value.
 \begin{lemma}\label{Lem3.delta} The function
 \begin{equation}\label{3.D}
 D(\lambda)=D_\xi(\lambda):=\frac{\|G_\lambda(\cdot,\xi)\|^2_A}{\|G_\lambda(\cdot,\xi)\|_B^2}
 \end{equation}
 is strictly increasing on $[-\lambda_0,\infty)$. Moreover,
 \begin{equation}\label{3.lims}
 \lim_{\lambda\to-\lambda_0}D(\lambda)=\lambda_0\ \ {\rm and}
 \ \ \lim_{\lambda\to+\infty}D(\lambda)=\infty,
 \end{equation}
 so the inverse function
 $D\to\lambda(D)$ is well defined on $[\lambda_0,\infty)$.
\end{lemma}
\begin{proof} We
will show that $D'(\lambda)>0$. To this end, we note that the function
$G_\lambda(x,\xi)$ is smooth with respect to $\lambda$  and
the derivative $G'_\lambda(x,\xi):=\partial_\lambda G_\lambda(x,\xi)$
solves the equation
\begin{equation}\label{2.diff}
AG'_\lambda(\cdot,\xi)+\lambda BG'_\lambda(\cdot,\xi)=
\Bbb A(\lambda)G'_\lambda(\cdot,\xi)=-B G_\lambda(\cdot,\xi),
\end{equation}
so that
$$
AG'_\lambda(\cdot,\xi)=-BG_\lambda(\cdot,\xi)-\lambda
BG'_\lambda(\cdot,\xi), \quad
G'_\lambda(\cdot,\xi)=-\mathbb{A}(\lambda)^{-1}B G_\lambda.
$$
Therefore,
\begin{multline}\label{3.long}
\frac12\|G_{\lambda}(\cdot,\xi)\|^4_B D'(\lambda)=\\=\|G_\lambda(\cdot,\xi)\|^2_B(AG'_\lambda(\cdot,\xi),G_\lambda(\cdot,\xi))-(BG'_\lambda(\cdot,\xi),G_\lambda(\cdot,\xi))\| G_\lambda(\cdot,\xi)\|^2_A=\\=-\|G_\lambda(\cdot,\xi)\|^4_{B}-\lambda\|G_\lambda(\cdot,\xi)\|^2_{B}(BG'_\lambda(\cdot,\xi),G_\lambda(\cdot,\xi))-
(BG'_\lambda(\cdot,\xi),G_\lambda(\cdot,\xi))\|G_\lambda(\cdot,\xi)\|^2_A=\\=
-\|G_\lambda(\cdot,\xi)\|^4_B-(BG'_\lambda(\cdot,\xi),G_\lambda(\cdot,\xi))(\Bbb A(\lambda) G_\lambda(\cdot,\xi),G_\lambda(\cdot,\xi))=\\=(\Bbb A^{-1}(\lambda) BG_\lambda(\cdot,\xi),BG_\lambda(\cdot,\xi))(\Bbb A(\lambda) G_\lambda(\cdot,\xi),G_\lambda(\cdot,\xi))-\|G_\lambda(\cdot,\xi)\|^4_B.
\end{multline}
By the Cauchy-Schwartz inequality we have
\begin{multline}\label{3.cauchy}
\|v\|_B^4=(Bv,v)^2=(\Bbb A^{-1/2}(\lambda)B v,\Bbb A^{1/2}(\lambda)v)^2
\le\\\le \|\Bbb A^{-1/2}(\lambda)Bv\|_H^2\|\Bbb A^{1/2}(\lambda)v\|_H^2=
(\Bbb A^{-1}(\lambda)Bv,Bv)(\Bbb A(\lambda)v,v)
\end{multline}
for all $v\in\bar H^m$. Taking $v(x)=G_\lambda(x,\xi)$, we see from \eqref{3.long} that $D'(\lambda)\ge0$. Moreover, equality in \eqref{3.cauchy} holds only if
$$
\alpha\Bbb A(\lambda)v=\beta Bv,\ \ \alpha,\beta\in\R
$$
and from elliptic regularity $v$ should be at least $C^\infty$-smooth.
However, $v=G_\lambda(\cdot,\xi)$ cannot be $C^\infty$-smooth, so the
equality is impossible and $D'(\lambda)>0$. Thus, we have proved that
$D(\lambda)$ is strictly increasing.

We now need to verify \eqref{3.lims}. We start with the limit
$\lambda\to-\lambda_0$. To solve \eqref{3.green} near
$\lambda=-\lambda_0$, we introduce
$w(x,\xi)=B^{1/2}G_\lambda(x,\xi)$. Then
\begin{equation}\label{3.eigen}
B^{-1/2}AB^{-1/2}w+\lambda w=B^{-1/2}\bar\delta(x,\xi)
\end{equation}
and, by definition, $\lambda_0$ is the smallest eigenvalue of the
self-adjoint positive operator $B^{-1/2}AB^{-1/2}$ in $\bar H$.
Clearly,  this operator has compact inverse, so its spectrum is
discrete. Let $\{\psi_k\}_{k=1}^\infty$ be family of the
orthonormal eigenfunctions of it, and let the first $p$ ($p\ge1$)
of them correspond to the smallest eigenvalue $\lambda_0$. Clearly,
$\psi_i\in C^\infty(\mathbb M)$. Seeking $w$ in the form
\begin{equation}\label{w-eqn}
w(x,\xi)=\sum_{i=1}^p a_i\psi_i(x)+\sum_{i=p+1}^\infty a_i\psi_i(x)
=:w_0+w^\bot,
\end{equation}
substituting this into~\eqref{3.eigen}, and taking the scalar
product with $\psi_1,\dots,\psi_p$, we find the coefficients of the
$w_0$-part of the solution:
$$
a_i=\frac1{\lambda+\lambda_0}(B^{-1/2}\bar\delta(\cdot,\xi),\psi_i)=
\frac{\tilde\psi_i(\xi)}{\lambda+\lambda_0},
\qquad i=1,\dots,p,
$$
where we set $\tilde\psi_i:=B^{-1/2}\psi_i$. Observe that the
$w^\bot$-part of the solution remains bounded as
$\lambda\to-\lambda_0$. Applying $B^{-1/2}$ to~\eqref{w-eqn} we
find $G_\lambda(x,\xi)$:
\begin{equation}\label{3.split}
G_\lambda(x,\xi)=\frac1{\lambda+\lambda_0}\sum_{i=1}^p\tilde\psi_i(\xi)\tilde\psi_i(x)+G^\bot_\lambda(x,\xi),
\end{equation}
where the part $G_\lambda^\bot(x,\xi)$ remains bounded as $\lambda\to-\lambda_0$. Thus,
$$
\lim_{\lambda\to-\lambda_0}D(\lambda)=\frac{\|\sum_{i=1}^p\tilde\psi_i(\xi)\tilde\psi_i(\cdot)\|_A^2}
{\|\sum_{i=1}^p\tilde\psi_i(\xi)\tilde\psi_i(\cdot)\|_B^2}=\lambda_0\frac{\sum_{i=1}^p\tilde\psi_i(\xi)^2}{\sum_{i=1}^p\tilde\psi_i(\xi)^2}=\lambda_0.
$$
Let us consider the case $\lambda\to\infty$. Assume that \eqref{3.lims} is wrong and we have
$$
\lim_{\lambda\to\infty}D(\lambda)=D_{max}<\infty.
$$
Then,
\begin{equation}\label{3.bound}
\|G_\lambda(\cdot,\xi)\|^2_A\le D_{max}\|G_\lambda(\cdot,\xi)\|^2_B.
\end{equation}
Multiplying equation \eqref{3.green} by $G_\lambda(x,\xi)$ integrating in $x\in\mathcal M$, and using the embedding $\bar H^m\subset C$ and \eqref{3.bound}, we get
$$
\|G_\lambda(\cdot,\xi)\|^2_A+\lambda\|G_\lambda(\cdot,\xi)\|^2_B=G_\lambda(\xi,\xi)\le\|G_\lambda(\cdot,\xi)\|_{C}\le C\|G_\lambda(\cdot,\xi)\|_A\le C_1\|G_\lambda(\cdot,\xi)\|_B.
$$
Due to this estimate, we have
$$
\|G_\lambda(\cdot,\xi)\|_{A}\le C\|G_\lambda(\cdot,\xi)\|_B\le C_2\lambda^{-1},
$$
where $C_2$ is independent of $\lambda\to\infty$. Thus,
$w_\lambda(x,\xi):=\lambda G_\lambda(x,\xi)$ is uniformly bounded in
$\bar H^m$ and, without loss of generality, we may assume that
$w_\lambda(\cdot,\xi)\to w_\infty(\cdot,\xi)$ weakly in this space.
Then, obviously, $w_\infty(\cdot,\xi)\in \bar H^m$ and
$$
Bw_\infty(\cdot,\xi)=\bar\delta(\cdot,\xi).
$$
In particular, since $w_\infty\in\bar H^m$, we have
$\bar\delta(\cdot,\xi)\in \bar H^{m-2l}$. Due to the assumption
\eqref{2.dim} we have $m-2l>-n/2$, so that $\delta(\cdot,\xi)\in H^{-s}$
for some $s<n/2$ which is impossible. Thus, \eqref{3.lims} is proved
and the lemma is also proved.
\end{proof}

\subsection{Main result} \label{ss4} The aim of this subsection is to give
the analytic expression for the function $\Bbb V$ in terms of the
Green's
functions introduced above. This result is stated in the following theorem.

\begin{theorem}\label{Th4.main} Let the above assumptions hold. Then, for
every $D\in[\lambda_0,\infty)$ and every $\xi\in\mathcal M$, the supremum
in \eqref{2.max} is the maximum and this maximum is achieved in a unique point
\begin{equation}\label{4.ext}
u_{D,\xi}(x):=\frac{G_{\lambda(D)}(x,\xi)}{\|G_{\lambda(D)}(\cdot,\xi)\|_B},
\end{equation}
where the function $\lambda(D)$
is defined in Lemma \ref{Lem3.delta}. In particular,
\begin{equation}\label{4.an}
\Bbb V(\xi,D)=\(\frac{G_{\lambda(D)}(\xi,\xi)}{\|G_{\lambda(D)}(\cdot,\xi)\|_B}\)^2.
\end{equation}
\end{theorem}
\begin{proof} Indeed, let $u\in \bar H^m$, be such that $\|u\|_{B}=1$ and $\|u\|_A^2=D$. Then, according to Lemma \ref{Lem3.delta}, there is a unique
$\lambda=\lambda(D)$ which solves \eqref{3.D}. According to \eqref{3.main} with $\lambda=\lambda(D)$,
$$
|u(\xi)|^2\le G_{\lambda(D)}(\xi,\xi)(D+\lambda(D))
$$
and the equality here holds if and only if
$u(x)=c G_{\lambda(D)}(x,\xi)$. Taking the $B$-norm of both sides of
this equality, we see that $c=\|G_{\lambda(D)}(\cdot,\xi)\|_B^{-1}$ and
$u(x)=u_{D,\xi}(x)$. This finishes the proof of the theorem.
\end{proof}
Thus, in order to find $\Bbb V$, we need three functions
\begin{equation}\label{4.fgh}
f_\xi(\lambda):=G_\lambda(\xi,\xi),\ \ g_\xi(\lambda):=\|G_\lambda(\cdot,\xi)\|^2_B,\ \ h_\xi(\lambda):=\|G_\lambda(\cdot,\xi)\|^2_A.
\end{equation}
Then
\begin{equation}\label{4.fgh1}
D(\lambda)=\frac{h_\xi(\lambda)}{g_\xi(\lambda)}\ \ {\rm and}\ \
 \Bbb V(\xi,\lambda):=\Bbb V(\xi,D(\lambda))=\frac{f_\xi(\lambda)^2}{g_\xi(\lambda)},\ \ \lambda\in[-\lambda_0,\infty)
\end{equation}
and we have the parametric representation of the function $D\to\Bbb V(\xi,D)$. The next lemma shows that
the functions $g$ and $h$ can be expressed in terms of $f$.
\begin{lemma}\label{Lem4.relations}
The functions $f$, $g$ and $h$ satisfy the following equalities
\begin{equation}\label{4.fgfh}
f'_\xi(\lambda):=\frac{d}{d\lambda}f_\xi(\lambda)=-g_\xi(\lambda),\ \ h_\xi(\lambda)=f_\xi(\lambda)+\lambda f'_\xi(\lambda).
\end{equation}
\end{lemma}
\begin{proof} Indeed, multiplying \eqref{3.green} by $G_\lambda(x,\xi)$
and integrating it over $x\in\mathcal M$, we have
\begin{equation}\label{4.green}
f_\xi(\lambda)=G_\lambda(\xi,\xi)=\|G_\lambda(\cdot,\xi)\|^2_A+\lambda\|G_\lambda(\cdot,\xi)\|^2_B=h_\xi(\lambda)+\lambda g_\xi(\lambda).
\end{equation}
Differentiating this formula with
respect to $\lambda$ and using \eqref{2.diff}, we get
\begin{multline}\label{4.diff}
f'_\xi(\lambda)=2(\Bbb A(\lambda)G'_\lambda(\cdot,\xi),G_\lambda(\cdot,\xi))+g_\xi(\lambda)=\\=
-2(BG_\lambda(\cdot,\xi),G_\lambda(\cdot,\xi))+g_\xi(\lambda)=-2g_\xi(\lambda)+g_\xi(\lambda)=-g_\xi(\lambda).
\end{multline}
It only remains to note that \eqref{4.diff} and \eqref{4.green} imply \eqref{4.fgfh} and finish the proof of the lemma.
\end{proof}
The next result shows that
the sharp constant in the inequality
\begin{equation}\label{4.ineq1}
|u(\xi)|^2\le K\|u\|_B^{2\theta}\|u\|_A^{2(1-\theta)},\ u\in\bar H^m.
\end{equation}
is expressed in terms of the following
\textit{scalar} maximization problem again involving the Green's function $G_\lambda(\xi,\xi)$.

\begin{theorem}\label{Th4.main1} For a fixed $\xi\in\mathcal{M}$ the
constant $K$  in \eqref{4.ineq1} is given by
\begin{equation}\label{C.sharp}
K=K(\xi):=\frac{1}{\theta^\theta(1-\theta)^{1-\theta}}\cdot
\sup_{\lambda>0}\bigg\{\lambda^{\theta}G_\lambda(\xi,\xi)\bigg\}<\infty,
\end{equation}
where $\theta$ is defined in \eqref{2.int}
and the constant \eqref{C.sharp} is sharp. Furthermore, the
extremal function in~\eqref{4.ineq1} exists if and only if
the supremum in~\eqref{C.sharp} is attained at a finite
point $\lambda_*$.
\end{theorem}
\begin{proof} We first show that \eqref{C.sharp} is finite. Indeed,
using \eqref{3.GG} and the inequality \eqref{2.int} (with the
non-optimal constant $C$!), we have
\begin{multline*}
\lambda^{\theta} G_\lambda(\xi,\xi)^2\le
\lambda^{\theta}\|G_\lambda(\xi,\cdot)\|_{L^\infty}^2\le
C(\lambda\|G_\lambda(\xi,\cdot)\|^2_B)^{\theta}
(\|G_\lambda(\xi,\cdot)\|_A^2)^{1-\theta}\le\\
\le C(\|G_\lambda(\xi,\cdot)\|_A^2+\lambda\|G_\lambda(\xi,\cdot)\|^2_B)
=C G_\lambda(\xi,\xi)
\end{multline*}
and \eqref{C.sharp} is finite.
\par
Let us check that \eqref{4.ineq1} holds with $K=K(\xi)$. Indeed, let
$u\in\bar H^m$ be arbitrary and let $\lambda:=\frac{\theta}{1-\theta}
\frac{\|u\|_A^2}{\|u\|_B^2}$.
Then, using \eqref{3.main}, we have
\begin{multline}\label{4.long}
|u(\xi)|^2\le G_\lambda(\xi,\xi)\|u\|_B^2\(\frac{\|u\|_A^2}{\|u\|_B^2}+
\lambda\)=\frac1\theta G_\lambda(\xi,\xi)\lambda\|u\|_B^2=
\frac1\theta G_\lambda(\xi,\xi)\lambda^\theta\lambda^{1-\theta}\|u\|_B^2=
\\=\frac1\theta\lambda^\theta G_\lambda(\xi,\xi)
\left(\frac{\theta}{1-\theta}\frac{\|u\|_A^2}{\|u\|_B^2}\right)^{1-\theta}
\|u\|_B^2= \frac{1}{\theta^\theta(1-\theta)^{1-\theta}}\cdot\lambda^{\theta}
G_\lambda(\xi,\xi)\|u\|_B^{2\theta}\|u\|^{2(1-\theta)}_A\le\\
\le \frac{1}{\theta^\theta(1-\theta)^{1-\theta}}\cdot
\sup_{\lambda>0}\bigg\{\lambda^{\theta} G_\lambda(\xi,\xi)\bigg\}\,
\|u\|_B^{2\theta}\|u\|^{2(1-\theta)}_A=
K(\xi)\|u\|_B^{2\theta}\|u\|^{2(1-\theta)}_A.
\end{multline}
Let us check that $K(\xi)$ is sharp. We first assume that the supremum in \eqref{C.sharp} is the maximum which is achieved at $\lambda=\lambda_*$.
Then,
$$
0=\frac{d}{d\lambda}(\lambda^\theta f_\xi(\lambda))|_{\lambda=\lambda_*}=\lambda^{\theta}_+(f'_\xi(\lambda_*)+\theta\lambda^{-1}_+f_\xi(\lambda_*))
$$
and $\lambda f'_\xi(\lambda_*)+\theta f_\xi(\lambda_*)=0$. Therefore,
due to \eqref{4.fgfh},
$$
D(\lambda_*)=\frac{h_\xi(\lambda_*)}{g_\xi(\lambda_*)}=
-\frac{f_\xi(\lambda_*)+\lambda_* f'_\xi(\lambda_*)}{f'_\xi(\lambda_*)}
=\frac{1-\theta}\theta\lambda_*.
$$
Thus, $\lambda_*=\frac{\theta}{1-\theta}
\frac{\|G_{\lambda_*}(\xi,\cdot)\|^2_A}
{\|G_{\lambda_*}(\xi,\cdot)\|^2_B}$ and all inequalities in
\eqref{4.long} become equalities if we take $u(x)=G_{\lambda_*}(\xi,x)$,
 so we have the exact extremal function in that case.
\par
Since $\lambda^\theta G_\lambda(\xi,\xi)\to0$ as $\lambda\to0$, we only need to consider the case when the supremum in \eqref{C.sharp} is achieved as $\lambda\to\infty$. Then, two alternative cases are possible:
\par
1) there are sequence $\{\lambda_k\}_{k=1}^\infty$ of {\it local} maximums such that
$$
\lambda_k^\theta G_{\lambda_k}(\xi,\xi)\to \sup_{\lambda>0}\{\lambda^\theta G_\lambda(\xi,\xi)\}
$$
Since the derivative vanishes at local maximums, then arguing as before, we see that the sequence of conditional extremals $u_n(x):=G_{\lambda_n}(\xi,x)$ does not allow us to take the constant $K$ strictly less than $K(\xi)$ and \eqref{C.sharp} is sharp.
\par
2) The function  $\lambda^\theta G_\lambda(\xi,\xi)$ is eventually
 monotone increasing  as $\lambda\to\infty$. Then the limit
\begin{equation}\label{4.Glim}
G_\infty:=\lim_{\lambda\to\infty}\lambda^\theta G_\lambda(\xi,\xi)
\end{equation}
exists and is strictly positive. Using the fact that the derivative is
integrable we can find sequences $\lambda_k\to\infty$ and $\eb_k\to0$
such that
$$
\frac {d}{d\lambda}(\lambda^\theta f_\xi(\lambda))|_{\lambda=\lambda_k}=\eb_k\lambda_k^{-1}.
$$
This, together with \eqref{4.Glim} gives
$$
\lambda_k^\theta f_\xi(\lambda_k)= G_\infty +o_{\lambda\to\infty}(1),
\ \ \lambda_k^{1+\theta}f'_\xi(\lambda_k)=
-\theta G_\infty+o_{\lambda\to\infty}(1).
$$
Therefore,
\begin{equation}\label{4.Dk}
\frac{D(\lambda_k)}{\lambda_k}=\frac{1-\theta}\theta+o_{\lambda\to\infty}(1).
\end{equation}
Finally, taking $u_k(x):=G_{\lambda_k}(\xi,x)$,
after  straightforward transformations we see that
\begin{equation}
\frac{|u_k(\xi)|^2}{\|u_k\|^{2\theta}_B\|u_k\|_A^{2(1-\theta)}}=
\(\lambda_k^\theta G_{\lambda_k}\)\cdot
\(\frac{\lambda_k}{D(\lambda_k)}\)^{1-\theta}
\cdot\(1+\frac{D(\lambda_k)}{\lambda_k}\)
\end{equation}
Passing to the limit $k\to\infty$ and using \eqref{4.Dk}, we obtain exactly
$K(\xi)$ in the right-hand side and verify that $K(\xi)$ is sharp in the
second case as well.

To complete the proof it remains to show that if there
exists an extremal function $u_*$ in~\eqref{4.ineq1},
\eqref{C.sharp}, then the supremum with respect to $\lambda$
in \eqref{C.sharp} is attained at a finite point.

Using the elementary identity for positive $a$, $b$
$$
\lambda_*^\theta a^{2(1-\theta)}b^{2\theta}=
\theta^\theta(1-\theta)^{(1-\theta)}\left(a^2+\lambda_* b^2\right),
\qquad\lambda_*=\frac\theta{1-\theta}\frac{a^2}{b^2},
$$
we have for the extremal function $u_*$ the equality
$$
\aligned
u_*(\xi)=K(\xi)\|u\|_B^{2\theta}\|u\|_A^{2(1-\theta)}=
K(\xi)\frac{\theta^\theta(1-\theta)^{(1-\theta)}}{\lambda_*^\theta}\left(
\|u_*\|_A^2+\lambda_*\|u_*\|^2_B\right)=\\=
\sup_{\lambda>0}\bigg\{\lambda^{\theta}G_\lambda(\xi,\xi)\bigg\}
\frac1{\lambda_*^\theta}
\left(\|u_*\|_A^2+\lambda_*\|u_*\|^2_B\right),\qquad\text{where}\quad
\lambda_*=\frac\theta{1-\theta}\frac{\|u_*\|_A^2}{\|u_*\|_B^2}.
\endaligned
$$
Lemma~\ref{Lem3.est} now gives that necessarily
$u_*(x)=G_{\lambda_*}(x,\xi)$ and
$$
\sup_{\lambda>0}\bigg\{\lambda^{\theta}G_\lambda(\xi,\xi)\bigg\}
\frac1{\lambda_*^\theta}=G_{\lambda_*}(\xi,\xi),\quad
\text{or}\quad
\sup_{\lambda>0}\bigg\{\lambda^{\theta}G_\lambda(\xi,\xi)\bigg\}=
\lambda_*^{\theta}G_{\lambda_*}(\xi,\xi).
$$
The proof is complete.
\end{proof}

\subsection{Asymptotic expansions for big $\lambda$}\label{ss4.5}
In this subsection, we derive some useful formulas for the function
$\Bbb V(\xi,D)$ when $D\to\infty$. To this end, we need to know the
asymptotic behavior of the Green's function
$f_\xi(\lambda)=G_\lambda(\xi,\xi)$. It is not difficult to see using
the localization and frozen coefficients technique that the limit
\eqref{4.Glim} exists and is strictly positive, so the leading term in
the asymptotic expansions of $G_\lambda$ is known. However, the further
terms in the asymptotic expansion seems problem dependent and we do not
know the general formulas for them. By this reason, we just assume that
\begin{equation}\label{6.expand}
f_\xi(\lambda)=\lambda^{-1}\(g_1 \lambda^{1-\theta}+g_2+g_3\lambda^{\theta-1}+o(\lambda^{\theta-1})\),
\end{equation}
where $g_1=g_1(\xi)>0$ and $g_2=g_2(\xi)$, $g_3=g_3(\xi)$ are some given
numbers, and that we are able to differentiate the expansions
\eqref{6.expand} with respect to $\lambda$. This assumption will be
satisfied in most part of our applications. Then, the following result holds.
\begin{proposition}\label{Prop6.as} Let the above assumptions hold and let,
 in addition, the asymptotic expansions \eqref{6.expand} be true. Then
the following asymptotic expansion holds as  $D\to\infty$:
\begin{equation}\label{6.mist}
\Bbb V(\xi,D)=g_1 S D^{1-\theta}+g_2\frac1\theta-
\frac12 S^{-1}\frac{g_2^2(1-\theta)-
2\theta g_1g_3}{\theta^3g_1} D^{\theta-1}+o(D^{\theta-1}) ,\ \
S:=\frac1{\theta^\theta(1-\theta)^{1-\theta}}.
\end{equation}
\end{proposition}
\begin{proof}
The proof of this proposition is a straightforward
(although rather technical) computation.
\ Setting for brevity $g_1=a$, $g_2=b$, $g_3=c$ we have
$$
\aligned
&f(\lambda)=a\lambda^{-\theta}+b\lambda^{-1}+c\lambda^{-(2-\theta)}
+o(\lambda^{-(2-\theta)}),\\
&g(\lambda)=a\theta\lambda^{-(1+\theta)}+b\lambda^{-2}+c(2-\theta)
\lambda^{-(3-\theta)}+o(\lambda^{-(3-\theta)}),\\
&h(\lambda)=f(\lambda)-\lambda g(\lambda)=
a(1-\theta)\lambda^{-\theta}-c(1-\theta)\lambda^{-(2-\theta)}+
o(\lambda^{-(2-\theta)}).
\endaligned
$$
Next, we find the asymptotics, as $D\to\infty$, of the
(unique) solution $\lambda=\lambda(D)$ of the first
equation in~\eqref{4.fgh1}:
$$
\aligned
D(\lambda)=\frac{h(\lambda)}{g(\lambda)}=
\frac{\lambda(a(1-\theta)-c(1-\theta)\lambda^{-(2-2\theta)}+\dots)}
{a\theta+b\lambda^{-(1-\theta)}+c(2-\theta)\lambda^{-(2-2\theta)}+\dots}=\\
\frac{1-\theta}\theta\left(\lambda-\frac
b{a\theta}\lambda^\theta+
\frac{b^2-2\theta
ac}{a^2\theta^2}\lambda^{-1+2\theta}+\dots\right),
\endaligned
$$
or
\begin{equation}\label{D(lambda)}
\lambda-A\lambda^\theta+B\lambda^{-1+2\theta}+\dots=\delta:=
D\frac\theta{1-\theta},
\qquad
A=\frac b{a\theta},\ \ B=\frac{b^2-2\theta ac}{a^2\theta^2}.
\end{equation}
The unique large solution of this equation
has the asymptotics as $\delta\to\infty$
$$
\lambda(\delta)=\delta+A\delta^\theta+C\delta^{-1+2\theta}+\dots,
$$
where we find $C$ by substituting the last expression
into~\eqref{D(lambda)}, which gives
$$
C=\theta
A^2-B=\frac{-b^2(1-\theta)+2ac\theta}{a^2\theta^2},
$$
or, finally,
\begin{equation}\label{lambda(D)}
\lambda(D)=rD+sD^\theta+tD^{2\theta-1}+\dots\,,
\end{equation}
where
$$
r=\frac\theta{1-\theta},\quad
s=\frac ba\frac1{\theta^{1-\theta}(1-\theta)^\theta},\quad
t=\frac{2ac\theta-(1-\theta)b^2}
{a^2\theta^{3-2\theta}(1-\theta)^{2\theta-1}}\,.
$$
It remains to substitute~\eqref{lambda(D)} into
$\mathbb{V}=f^2/g=D\cdot f^2/h$, for which we have the expansion
\begin{equation}\label{f^2/h}
\frac{f(\lambda)^2}{h(\lambda)}=\frac
a{1-\theta}\lambda^{-\theta}+\frac{2b}{1-\theta}\lambda^{-1}+
\frac{b^2+3ac}{(1-\theta)a}\lambda^{\theta-2}+\dots\,.
\end{equation}
For each power we obtain from~\eqref{lambda(D)}, respectively,
$$
\aligned
&\lambda(D)^{-\theta}=\frac1{r^\theta}D^{-\theta}-
\frac{\theta s}{r^{1+\theta}}D^{-1}+
\frac{\theta(\theta+1)s^2-2\theta
tr}{2r^{2+\theta}}D^{\theta-2}+\dots\,,\\
&\lambda(D)^{-1}=\frac 1rD^{-1}-\frac
s{r^2}D^{\theta-2}+\dots\,,\\
&\lambda(D)^{\theta-2}=\frac1{r^{2-\theta}}D^{\theta-2}+\dots\,.
\endaligned
$$
Substituting this into~\eqref{f^2/h}, multiplying by $D$,
we obtain after quite a few miraculous cancellations
the asymptotic expansion~\eqref{6.mist}.
\end{proof}

\begin{remark}\label{Rem6.good} We see that, in particular,
if for some $\xi\in\mathcal M$ the third term in\eqref{6.mist} is negative
(this is always the case when $g_3\le0$), then
\begin{equation}\label{6.big}
\Bbb V(\xi,D)<g_1 SD^{1-\theta}+\frac{g_2}\theta,
\end{equation}
for large $D$. Then, using the numerics which is reliable for relatively
small $D$, we will show that, in some cases, inequality \eqref{6.big}
holds for {\it all} values of $D$.
This will give us the improved version of \eqref{2.int}:
\begin{equation}
|u(\xi)|^2\le g_1 S\|u\|_B^\theta\|u\|_A^{1-\theta}+\frac{g_2}\theta\|u\|^2_B
\end{equation}
with best possible constants.

We also mention an interesting fact that
$$
D'(\lambda)=\(\frac{f_\xi(\lambda)+\lambda f'_\xi(\lambda)}{-f'_\xi(\lambda)}\)'=\frac{2(f_\xi'(\lambda))^2-f''_\xi(\lambda)f(\lambda)}{(f'_\lambda(\xi))^2}
$$
and, therefore, the positivity of $D'(\lambda)$ proved in
Lemma~\ref{Lem3.delta} is equivalent to the strict convexity of the
function $\lambda\to\frac1{f_\xi(\lambda)}$.
\end{remark}
\subsection{Variational characterisation of $\Bbb V(\xi,\lambda)$} In
this subsection, we give a simple, but very useful description of
$\Bbb V(\xi,\lambda)$ in terms of the Green's function
$G_\lambda(\xi,\xi)$ which does not involve the derivatives in
$\lambda$ and which allows us in many cases to prove
inequality \eqref{6.big} {\it analytically} for all admissible values
of $D$. Namely, the following theorem holds.
\begin{theorem}\label{Th.var} Let the assumptions of Theorem
\ref{Th4.main} hold. Then, the function $\Bbb V(\xi,D)$ defined in
\eqref{4.an} can be expressed as follows:
\begin{equation}\label{green.var}
\Bbb V(\xi,D)=\inf_{\lambda\in[-\lambda_0,\infty)}\{(\lambda+D)G_\lambda(\xi,\xi)\},
\end{equation}
where $D\in[\lambda_0,\infty)$ and the Green's function
$G_\lambda(\xi,\xi)$ is defined in \eqref{3.green}.
\end{theorem}
\begin{proof} We first note that, due to the results obtained above the
function $\lambda\to(\lambda+D)G_\lambda(\xi,\xi)$ tends to $+\infty$
when
$\lambda\to\infty$ or $\lambda\to-\lambda_0$ for every
$D\in(\lambda_0,\infty)$. Thus, the infimum in \eqref{green.var} is
achieved at some point $\lambda_0=\lambda_0(D)$ inside  the interval.
Obviously, this $\lambda_0$ solves the equation
$$
G_\lambda(\xi,\xi)+(D+\lambda)\frac d{d\lambda}G_\lambda(\xi,\xi)=0,\
\text{or}\ \ D=-\frac{\lambda\frac d{d\lambda}G_\lambda(\xi,\xi)+
G_\lambda(\xi,\xi)}{\frac d{d\lambda}G_\lambda(\xi,\xi)}=
\frac{h_\xi(\lambda)}{g_\xi(\lambda)},
$$
see \eqref{4.fgh}, \eqref{4.fgh1} and \eqref{4.fgfh}. Thus, the equation on
the minimal value $\lambda_0(D)$ coincides with equation \eqref{3.D}
for $\lambda(D)$ which has a unique solution and, therefore, the minimum
in \eqref{green.var} is achieved exactly at $\lambda_0=\lambda(D)$. It
remains to check that the value at this point is exactly
$\Bbb V(\xi,\lambda)$. To this end, we note that at the extremal point
the numbers $D$ and $\lambda$ satisfy the first equation of
\eqref{4.fgh1} and, at that point
$$
(\lambda+D)G_\lambda(\xi,\xi)=\frac{f_\lambda(\xi)(h_\xi(\lambda)+
\lambda g_\xi(\lambda))}{g_\xi(\lambda)}=
\frac{f_\xi(\lambda)^2}{g_\xi(\lambda)}=\Bbb V(\xi,D),
$$
where we have used~\eqref{4.fgh1}. Theorem \ref{Th.var} is proved.
\end{proof}
\subsection{Generalizations}\label{ss5} Here we briefly discuss several
possibilities to relax the assumptions on the manifold $\mathcal M$ and
operators $A$ and $B$.
\par
{\it I)} The operators $A$ and $B$ should not necessarily  be elliptic
{\it differential} operators. All the theory works word for word if we
assume that $A$ and $B$ are elliptic {\it pseudo-differential} operators
(for instance, $A=(-\Delta)^m$ and $B=(-\Delta)^l$ where $\Delta$ is the
Laplace-Beltrami operator on $\mathcal M$) and, in particular, the
numbers $m$ and $l$ may be not integers. This allow us to study the
interpolation inequalities in fractional Sobolev spaces.
\par
{\it II)} The theory can be naturally extended to the case where the
manifold $\mathcal M$ has a boundary. In this case, we need to assume
that the elliptic operators $A$ and $B$ are endowed by the proper
{\it boundary conditions} and these boundary conditions are chosen in
such way that
\begin{equation}\label{5.domain}
\mathcal D(A)\subset\mathcal D(B),
\end{equation}
where $\mathcal D(A)$ and $\mathcal D(B)$ are the domains of the operators
$A$ and $B$. It is not difficult to see that under this extra
assumption(s) the above developed theory remains true for manifolds
with boundary as well.
\par
{\it III)} The spaces $\bar H^m$ may be further restricted, for instance,
we may consider not all functions, say, in a disk,  but only radially
symmetric ones. If the operators $A$ and $B$ are also radially symmetric, all the theory works in this case as well.
\par
{\it IV)} The above theory works in many cases where the manifold
$\mathcal M$ is {\it not compact}. The only problem here is that,
unlike the compact case, the Green's functions $G_\lambda(x,\xi)$
may have bad behavior as $x,\xi\to\infty$ and as a result, the integrals
used above may not have sense. So, in general theory, one should be
accurate with the extra assumptions on the non-compact part of
$\mathcal M$ and operators $A$ and $B$ as well as with the possible
continuous spectrum of $A$ and $B$. However, in all our ``non-compact"
applications, these questions will be obvious and transparent, so in
order to avoid the technicalities, we do not present here
any ``general theory" for the non-compact case.

\section{Part II. Examples and applications} \label{s2s1}

\subsection{The case of $\mathbb{R}^n$}\label{Rn}
To  illustrate our method we consider the
simplest case when $M=\mathbb{R}^n$, and let
$l$ and   $m$ satisfy
$-\infty<l<n/2<m<\infty$, so that $0<\theta=\frac{2m-n}{2(m-l)}<1$.

\begin{theorem}\label{T:R^n}
The following inequality holds
\begin{equation}\label{R^n}
\|u\|_\infty^2\le c_{\mathbb{R}^n}(l,m)
\|(-\Delta)^{l/2}u\|^{2\theta}\|(-\Delta)^{m/2}u\|^{2(1-\theta)},
\end{equation}
where the sharp constant $c_{\mathbb{R}^n}(l,m)$ is
\begin{equation}\label{cR^n}
c_{\mathbb{R}^n}(l,m)=
(2\pi)^{-n}\int_{\mathbb{R}^n}\frac{d\xi}{\theta|\xi|^{2l}+
(1-\theta)|\xi|^{2m}}=
\frac{\sigma(n)\pi}
{(2\pi)^n2(m-l){\theta}^{\theta}(1-\theta)^{1-\theta}
\sin\pi{\theta}}\,.
\end{equation}
\end{theorem}
\begin{proof}
Using the Fourier transform
$
\widehat{f}(\xi)=\mathcal{F}(f)(\xi)=
(2\pi)^{-n/2}\int_{\mathbb{R}^n} e^{-i\xi x}f(x)dx,
$
for $\lambda>0$
\begin{equation}\label{Fourier}
\aligned
\|u\|_\infty^2\le(2\pi)^{-n}\left(\int_{\mathbb{R}^n}
|\widehat{u}(\xi)|d\xi\right)^2\le
\\
(2\pi)^{-n}\int_{\mathbb{R}^n}
|\widehat{u}(\xi)|^2\left(\lambda\theta|\xi|^{2l}+(1-\theta)|\xi|^{2m}\right)
d\xi
\int_{\mathbb{R}^n}\frac{d\xi}{\lambda\theta|\xi|^{2l}+(1-\theta)|\xi|^{2m}}=
\\
(2\pi)^{-n}\int_{\mathbb{R}^n}\frac{d\xi}{\lambda\theta|\xi|^{2l}+
(1-\theta)|\xi|^{2m}}
\cdot\left(\lambda\theta\|(-\Delta)^{l/2}u\|^2+
(1-\theta)\|(-\Delta)^{m/2}u\|^2\right).
\endaligned
\end{equation}
Next, setting $\lambda=\lambda_*=\|(-\Delta)^{m/2}u\|^2/\|(-\Delta)^{l/2}u\|^2$
we have
\begin{equation}\label{lam*Rn}
\lambda_*\theta\|(-\Delta)^{l/2}u\|^2+(1-\theta)\|(-\Delta)^{m/2}u\|^2=
\lambda_*^{\theta}\|(-\Delta)^{l/2}u\|^{2\theta}
\|(-\Delta)^{m/2}u\|^{2(1-\theta)},
\end{equation}
which gives
$$
\|u\|_\infty^2\le(2\pi)^{-n}
\left[\lambda_*^{\theta}
\int_{\mathbb{R}^n}\frac{d\xi}{\lambda_*\theta|\xi|^{2l}+(1-\theta)|\xi|^{2m}}\right]
\|(-\Delta)^{l/2}u\|^{2\theta}\|(-\Delta)^{m/2}u\|^{2(1-\theta)}.
$$
The expression in brackets is, in fact, independent of
$\lambda_*$, which gives~(\ref{cR^n}).
 The fact that the constant is sharp can be verified by
 substituting
\begin{equation}\label{extr-Rn}
u_\lambda(x)=\mathcal{F}^{-1}\left((\lambda\theta_1|\xi|^{2l}
+\theta_2|\xi|^{2m})^{-1}\right),
\end{equation}
and calculating the corresponding integrals. A simpler way,
however, is to observe first that for
$u_\lambda(x)$ all inequalities in (\ref{Fourier}) become
equalities. Next, differentiating the expression in  brackets
with respect to  $\lambda$, one can see that
$
\lambda=\|(-\Delta)^{l/2}u_\lambda\|^2/\|(-\Delta)^{m/2}u_\lambda\|^2
$,
which proves that~(\ref{R^n}) for $u=u_\lambda$
becomes an equality.
\end{proof}
\begin{remark}\label{R:G_lambda}
{\rm
Of course, we get the same result by applying
Theorem~\ref{Th4.main}. Here
$$
\mathbb{A}(\lambda)=(-\Delta)^m+\lambda(-\Delta)^l,\qquad
G_\lambda(x,\xi)=
\frac1{(2\pi)^n}
\int_{\mathbb{R}^n}\frac{e^{i\eta(x-\xi)}\,d\eta}
{|\eta|^{2m}+\lambda|\eta|^{2l}}
$$
and
$$
G_\lambda(\xi,\xi)=\lambda^{-\theta}
\frac{\sigma(n)\pi}
{(2\pi)^n2(m-l)\sin\pi\theta}\,.
$$

}
\end{remark}

\begin{corollary}
For any $D>0$ the maximization problem
\begin{equation}\label{Rn.max}
\mathbb V(D):=\sup\bigg\{|u(0)|^2:\ \ \ \|(-\Delta)^{l/2}u\|^2=1,
\ \ \|(-\Delta)^{m/2}u\|^2=D\bigg\}
\end{equation}
has the solution
\begin{equation}\label{Rn-sol}
\mathbb V(D)=c_{\mathbb{R}^n}(l,m)D^{1-\theta}.
\end{equation}
 The unique extremal function is
$$
U_D=
\frac{u_D}{\|(-\Delta)^{l/2}u_D\|}\,,
$$
see~\eqref{extr-Rn}.
\end{corollary}

\begin{remark}\label{R:int-beta}
{\rm
The integral was calculated using the formula
\begin{equation}\label{int-beta}
\int\limits_0^\infty\frac {x^m}{(1+x^k)^l}dx\,=\,\frac 1k
B\left(\frac {m\!+\!1}k,\,l\!-\!\frac {m\!+\!1}k\right),
\end{equation}
which will also be helpful in what follows.
}
\end{remark}
\begin{remark}\label{R:Taikov}
{\rm
In the 1D case the constant $c_{\mathbb{R}}(l,m)$ was found
in~\cite{Taikov}.
}
\end{remark}

\subsection{Symmetric manifolds}\label{s2s2}

In this section, we discuss the case where the underlying manifold $\Cal M$
is {\it symmetric} and the operators $A$ and $B$ are invariant with respect
to the symmetry group.
As a consequence of this, the Green's function $G_\lambda(x,\xi)$ introduced
 in \eqref{3.green} depends only on $x-\xi$
 (since every two points on $\mathcal M$ can be identified by the proper
 symmetry map). By this reason, the key function $G_\lambda(\xi,\xi)$ is in
 fact independent of $\xi$ and is a function of one variable $\lambda$:
\begin{equation}\label{3.symmetry}
G_\lambda(\xi,\xi)=G(\lambda).
\end{equation}
This observation simplifies greatly the analysis and allows us in many cases to compute explicitly the best constants in the appropriate interpolation inequalities. We restrict ourselves below only to consider the two model examples: tori and spheres although the developed technique is applicable to other symmetric manifolds as well.

\subsubsection{The tori} We start with the case of $n$-dimensional torus $\Bbb T^n:=[-\pi,\pi]^n$ and the inequalities of the form
\begin{equation}\label{tor.1}
\|u\|_{L_\infty(\mathbb T^n)}\le C_{l,n,m}
\|(-\Delta)^{l/2}u\|_{L_2(\Bbb T^n)}^\theta\|(-\Delta)^{m/2}u\|^{1-\theta}_{L_2(\Bbb T^n)}
\end{equation}
for the periodic functions $u\in H^{m}(\Bbb T^n)$ with zero mean
($\<u\>:=\int_{\Bbb T^n}u(x)\,dx=0$). In that case, $A:=(-\Delta)^m$,
$B=(-\Delta)^l$, $\operatorname{ker}A=\operatorname{ker} B=\{\mathrm{const}\}$ and
all of the assumptions of the above developed abstract
theory are satisfied if $l<\frac n2<m$. In particular,
\begin{equation}\label{tor.2}
\bar H^m=H^m(\Bbb T^n)\cap\{\<u\>=0\}.
\end{equation}
The Green's function of the operator
$\Bbb A(\lambda)=(-\Delta)^m+\lambda(-\Delta)^l$ on the torus $\Bbb T^n$
(with zero mean) can be found by expanding
it into the multi-dimensional Fourier series. This gives
\begin{equation}\label{tor.green}
G_\lambda(x,\xi)=\frac1{(2\pi)^n}\sum_{k\in\mathbb{Z}^n_0}\frac{e^{ik(x-\xi)}}{|k|^{2m}+\lambda|k|^{2l}},\
\
G_\lambda(\xi,\xi)=
G(\lambda)=\frac1{(2\pi)^n}\sum_{k\in\mathbb{Z}^n_0}\frac1{|k|^{2m}+\lambda|k|^{2l}},
\end{equation}
where $k=(k_1,\cdots,k_n)$ is the multi-index,
$|k|^2=k_1^2+\cdots+ k_n^2$ and the summation holds for all
multi-indexes $k=(k_1,\cdots,k_n)\ne(0,\cdots,0)$: $\mathbb{Z}^n_0=
\mathbb{Z}^n\setminus\{0\}$.

\paragraph{ The case $l=0$}\mbox{}
The particular case $l=0$ has been studied in \cite{Zelik}. In that case,
the asymptotic behavior of $G(\lambda)$ as $\lambda\to\infty$ can be
analyzed in a straightforward way  using the Poisson summation formula
(see, e.\,g., \cite{S-W}):
\begin{equation}\label{Poisson}
\sum_{m\in\mathbb{Z}^n}f(m/\mu)=
(2\pi)^{n/2}\mu^n
\sum_{m\in\mathbb{Z}^n}\widehat{f}(2\pi m \mu),
\end{equation}
where
$\mathcal{F}(f)(\xi)=\widehat{f}(\xi)=(2\pi)^{-n/2}\int_{\mathbb{R}^n}
f(x)e^{-i\xi\cdot x}dx$ is the Fourier transform and $\mu>0$.

Namely, let
$$
\Phi_\lambda(x)=(2\pi)^{-n/2}\mathcal{F}^{-1}
(1/({|\xi|^{2m}+\lambda}))=\frac1{(2\pi)^n}\int_{\R^n}
\frac{e^{i\xi\cdot x}\,d\xi}{|\xi|^{2m}+\lambda}
$$
be the fundamental solution of the differential operator
$\Bbb A(\lambda)$ in the whole space $\R^n$. Then~\eqref{Poisson} gives
\begin{equation}\label{tor.poi}
G(\lambda)+\frac1{(2\pi)^n}\lambda^{-1}=
\frac1{(2\pi)^n}\sum_{k\in\Bbb Z^n} \frac1{|k|^{2m}+\lambda}=
\sum_{k\in \Bbb Z^n}\Phi_\lambda(2\pi k).
\end{equation}
Furthermore, due to scaling invariance, $\Phi_\lambda(x)=
 \lambda^{-1}\lambda^{n/(2m)}\Phi_1(\lambda^{1/(2m)}x)$ and, due to the
 analyticity of the function
 $\xi\to\frac1{|\xi|^{2m}+1}$, we have $|\Phi_1(x)|\le Ce^{-C_m|x|}$
for some positive $C$ and $C_m$. Therefore, \eqref{tor.poi} reads
\begin{equation}\label{tor.poi1.0}
G(\lambda)=\lambda^{-1}\lambda^{n/(2m)}\Phi_1(0)-\frac1{(2\pi)^n}\lambda^{-1}+O(e^{-C_m\lambda^{1/(2m)}}).
\end{equation}
Finally, computing the table integral
$\Phi_1(0)=(2\pi)^{-n}\int_{\R^n}\frac{d\xi}{|\xi|^{2m}+1}$, we end up with
\begin{equation}\label{tor.Poisson}
G(\lambda)=(2\pi)^{-n}\lambda^{-1}\(\frac{\pi\sigma(n)}{2m\sin(\frac{\pi n}{2m})}\lambda^{\frac{n}{2m}}-1\)+O(e^{-C_{m}\lambda^{1/(2m)}}),
\end{equation}
where $\sigma(n)={2\pi^{n/2}}/{\Gamma(n/2)}$ is the surface area of the
$(n-1)$-dimensional unit sphere. Note that $1-\theta=n/{2m}$ now, so
\eqref{tor.Poisson} has the form of \eqref{6.expand} with $g_3=0$ and,
due to \ref{Prop6.as}, we have
\begin{equation}\label{tor.V}
\mathbb V(D)= c_n(m)D^{n/(2m)}-k_n(m)-l_n(m)D^{-n/(2m)}+O(D^{-n/m}),
\end{equation}
where
\begin{multline}\label{tor.const}
c_n(m):=c_{\mathbb{R}^n}(0,m)=\frac{\pi\sigma(n)}{(2\pi)^n n^{n/(2m)}(2m-n)^{1-n/(2m)}
\sin\frac{\pi n}{2m}},\ \
k_n(m):=\frac{2m}{(2\pi)^n(2m-n)},\\
l_n(m):=\frac{2n^{1+n/(2m)}m^2\sin\frac{\pi n}{2m}}{(2\pi)^n\pi\sigma(n)(2n-m)^{2+n/(2m)}},
\end{multline}
see \cite{Zelik} for the details. Thus, since $l_n(m)>0$,
\begin{equation}\label{tor.big}
\mathbb V(D)\le c_n(m)D^{n/(2m)}-k_n(m)
\end{equation}
for large $D$ and using the fact that $D\to\Bbb V(D)$ is continuous, we conclude that
\begin{equation}\label{tor.Kn}
K_n(m):=\sup_{D\in[1,\infty)}\{c_n(m)D^{n/(2m)}-\mathbb V(D)\}<\infty
\end{equation}
and
\begin{equation}\label{tor.all}
\mathbb V(D)\le c_n(m)D^{n/(2m)}-K_n(m)
\end{equation}
already for all $D\ge1$. Thus, according to \eqref{2.Vint}, we have established the following result of \cite{Zelik}.
\begin{theorem}\label{Th.torall} The following inequality holds for all
$u\in \bar H^m(\Bbb T^n)$:
\begin{equation}\label{tor.int}
\|u\|_{L_\infty(\Bbb T^n)}^2\le c_n(m)\|u\|_{L_2(\Bbb T^n)}^{2-n/m}
\|(-\Delta)^{m/2}u\|_{L_2(\Bbb T^n)}^{n/m}-K_n(m)\|u\|^2_{L^2(\Bbb T^n)},
\end{equation}
where $c_n(m)$ and $K_n(m)\le k_n(m)$ are sharp and are determined by \eqref{tor.const} and \eqref{tor.Kn} respectively.
\end{theorem}
We see that the constant $c_n(m)$ here coincides with the best constant
in the analogous inequality on the whole $\R^n$, see~\eqref{cR^n}.
 In contrast to that, the sharp constant $K_n(m)$ is unlikely to
be expressed analytically for all $n$ and $m$ since, according to the
definition \eqref{tor.Kn}, we need to know the function $V(D)$ not only
for large $D$, but for all $D\ge1$ in order to compute it. As shown in
\cite{Zelik}, the asymptotic expansion \eqref{tor.V} works for
{\it very large} $D$ only (if $m$ is large enough) and for the
intermediate values of $D$ this function is {\it oscillatory}
(which can be explained by studying the limit $m\to\infty$, see
\cite{Zelik} for the details).

Moreover, there is a strong difference between the case $n=1$ and the
multi-dimensional case $n>1$. In the first case, as shown in
\cite{Zelik}, we always have $K_1(m)>0$, so the lower order term in
\eqref{tor.int} {\it improves} the classical interpolation inequality.
However, in the multi-dimensional case, this constant become strictly
negative for sufficiently large $m$
(for instance, for $m>9$ if $n=2$ and for $m>6$ for $n=3$). Therefore,
in that case, the lower order corrector becomes {\it positive} and
necessary for the validity of the interpolation inequality. In other
words, if we are interested only in the classical interpolation
inequality on the torus (without the lower order correctors), we have
to {\it increase} the constant in comparison with the case of $\R^n$.

Nevertheless, as shown in \cite{Zelik} with the help of {\it numerics}
there are 3 particular cases where $K_n(m)=k_n(m)$ and as a consequence,
all constants in \eqref{tor.int} can be found analytically. That are

1) The case  $n=1$ and $m=1$: the inequality
\begin{equation}\label{tor.1.1}
\|u\|_{L_\infty(\Bbb T^1)}^2\le\|u\|_{L_2(\Bbb T^1)}\|u'\|_{L_2(\Bbb T^1)}-\frac1{\pi}\|u\|^2_{L_2(\Bbb T^1)}
\end{equation}
holds for all $2\pi$-periodic functions with zero mean;

2) The case $n=1$ and $m=2$: the inequality
\begin{equation}\label{tor.1.2}
\|u\|_{L_\infty(\Bbb T^1)}^2\le \frac{\sqrt2}{\sqrt[4]{27}}\|u\|_{L_2(\Bbb T^1)}^{3/2}
\|u''\|_{L_2(\Bbb T^1)}^{1/2}-\frac2{3\pi}\|u\|^2_{L_2(\Bbb T^1)}
\end{equation}
holds for all $2\pi$-periodic functions with zero mean;

3) The case $n=2$ and $m=2$: the inequality
\begin{equation}\label{tor.2.2}
\|u\|^2_{L_\infty(\Bbb T^2)}\le\frac14\|u\|_{L_2(\Bbb T^2)}\|\Delta u\|_{L_2(\Bbb T^2)}-\frac1{2\pi^2}\|u\|^2_{L_2(\Bbb T^2)}
\end{equation}
holds for all $2\pi\times 2\pi$-periodic functions with zero mean.
As the numerics suggests, that are the only cases (at least with integer  $m$) when the constant $K_n(m)=k_n(m)$, see \cite{Zelik}.

Based on the technique developed above, we give below the purely
analytic proof of the first two inequalities.
\begin{theorem}\label{Th.tor.1.1} The interpolation inequalities
\eqref{tor.1.1} and \eqref{tor.1.2} hold for all $2\pi$-periodic functions
with zero mean and all constants in those inequalities are sharp.
\end{theorem}
\begin{proof} We first consider~\eqref{tor.1.1}.
Since
$\Phi_\lambda(x)=(2\pi)^{-1/2}\mathcal{F}^{-1}((\xi^2+\lambda)^{-1})=
\frac1{2\sqrt{\lambda}}e^{-\sqrt\lambda |x|}$ for all $\lambda>0$,
according to \eqref{tor.poi} and summing the geometric progression,
we have
\begin{equation}\label{tor.expan}
G(\lambda)=\sum_{k\in\Bbb Z}
\frac1{2\sqrt{\lambda}}e^{-2\pi\sqrt\lambda |k|}-\frac1{2\pi\lambda}=
\frac1{2\pi}\frac{\pi\sqrt\lambda\coth(\pi\sqrt\lambda)-1}\lambda.
\end{equation}
Thus, in view of Theorem~\ref{Th.var}, for $D\ge1$,
\begin{equation}\label{tor.var}
\Bbb V(D)=\min_{\lambda\ge-1}\{(\lambda+D)G(\lambda)\}\le
\{(\lambda+D)G(\lambda)\}\vert_{\lambda=D-1/2}=(2D-1/2)G(D-1/2),
\end{equation}
where we have replaced minimum with respect to $\lambda$ by the value
at $\lambda=D-1/2$. Thus, we only need to prove that
\begin{multline}\label{tor.bigg}
(2D-1/2)G(D-1/2)-D^{1/2}+\frac1\pi=\\
=\frac1{4\pi(2D-1)}\(4\pi D\sqrt{4D-2}\cdot\coth\alpha
\,-\pi\sqrt{4D-2}\cdot\coth\alpha-2-
8\pi D^{3/2}+4\pi\sqrt{D}\)<0
\end{multline}
for all $D\ge1$, where we set for brevity
$\alpha:=\frac{\pi\sqrt{4D-2}}2$. To simplify the expression on the
right-hand side of \eqref{tor.big}, we use that
$$
\sqrt{4D-2}<2\sqrt{D}-\frac12 D^{-1/2}-\frac1{16}D^{-3/2},\ \ D\ge1
$$
(this inequality is obtained by expanding $\sqrt{1-(2D)^{-1}}$ in
Taylor series and noting that all dropped out terms there are negative).
Using also that $\coth\alpha\ge1$ in the negative terms, we end up with
\begin{multline}\label{tor.bigger}
\Bbb V(D)-D^{1/2}+\frac1{\pi}\le
\frac1{4\pi(2D-1)}\(4\pi D(8D^{1/2}-2D^{-1/2}-D^{-3/2}/4)\cdot
\coth\alpha-\right.\\\left.
\pi\sqrt{4D-2}\cdot\coth\alpha-2-8\pi D^{3/2}+4\pi\sqrt{D}\)=\\
=\frac1{16\pi D^{1/2}(2D-1)}\(32\pi D^2(\coth\alpha-1)-\right.\\\left.
8\pi D\coth\alpha-\pi\coth\alpha-4\pi\sqrt{4D-2}\sqrt{D}-8\sqrt{D}+
16\pi D\)<\\<
\frac1{16\pi D^{1/2}(2D-1)}\(\frac{64\pi D^2}{e^{2\alpha}-1}-
8\pi D-\pi-4\pi\sqrt{4D-2}\sqrt{D}-8\sqrt{D}+16\pi D\)=\\=
\frac1{16\pi D^{1/2}(2D-1)}\(\frac{64\pi D^2}{\exp(\pi\sqrt{4D-2})-1}+
\frac{4\pi}{1+\sqrt{1-(2D)^{-1}}}-\pi-8\sqrt D\).
\end{multline}
We note that the function $f(D):=\frac{64\pi D^2}{\exp(\pi\sqrt{4D-2})-1}$ is strictly decreasing when $D\ge1$ since
$$
f'(D)=-\frac{128\pi D\(\sqrt{4D-2}-\sqrt{4D-2}\exp(\pi\sqrt{4D-2})+\pi D\exp(\pi\sqrt{4D-2})\)}{(-1+\exp(\pi\sqrt{4D-2}))^2\sqrt{4D-2}}<0
$$
for all $D\ge1$. Analogously, the second term in the right-hand side of \eqref{tor.bigger} and the last one are also strictly decreasing, so replacing them by the their maximal values at $D=1$, we finally have
\begin{multline*}
\Bbb V(D)-D^{1/2}+\frac1{\pi}<
\frac1{16\pi D^{1/2}(2D-1)}\(\frac{64\pi}{\exp(\pi\sqrt{2})-1}+
\frac{8\pi}{2+\sqrt{2}}-\pi-8\)=\\=\frac1{16\pi D^{1/2}(2D-1)}
\cdot(-1.3873\ldots)<0.
\end{multline*}
Thus, inequality \eqref{tor.1.1} is proved. The fact that all constants
there are sharp follows from the previously established
asymptotics~\eqref{tor.V}:
$$
\Bbb V(D)=D^{1/2}-\frac1{\pi}+O(D^{-1/2})\quad\text{as}\quad D\to\infty.
$$
Before we turn to~\eqref{tor.1.2} we note that
the general formula~\eqref{lambda(D)} gives in our case
$
\lambda(D)=D-\frac2\pi\sqrt{D}+\dots\,.
$
Other choices of $\lambda$ in the substitution in~\eqref{tor.var}
will do, for example, $\lambda=D-1$, $\lambda=D-\frac2\pi\sqrt{D}$.
However, the choice $\lambda=D$ will not.

We now consider~\eqref{tor.1.2}. Using contour integration
(see \cite[Section 3.3]{Titch}) we can sum the series in
$G(\lambda)$:
$$
\sum_{n=1}^\infty\frac1{\mu^4+n^4}=
\frac{\pi\sqrt{2}}{4\mu^3}\,\frac{\sinh(\pi\sqrt{2}\mu)+\sin(\pi\sqrt{2}\mu)}
{\cosh(\pi\sqrt{2}\mu)-\cos(\pi\sqrt{2}\mu)}-\frac1{2\mu^4},
$$
and hence
\begin{equation}\label{G_lam.1.2}
G(\lambda)=\frac2{2\pi}\sum_{n=1}^\infty\frac1{n^4+\lambda}=
\frac1\pi\left(\frac{\pi\sqrt{2}}{4\lambda^{3/4}}\,
\frac{\sinh(\pi\sqrt{2}\lambda^{1/4})+\sin(\pi\sqrt{2}\lambda^{1/4})}
{\cosh(\pi\sqrt{2}\lambda^{1/4})-\cos(\pi\sqrt{2}\lambda^{1/4})}-
\frac1{2\lambda}\right).
\end{equation}
Since
$$
\frac{\sinh\alpha+\sin\alpha}
{\cosh\alpha-\cos\alpha}\le\frac{\sinh\alpha+1}
{\cosh\alpha-1}<1+4.1e^{-\alpha}\quad\text{for}\quad
\alpha=\pi\sqrt{2}\lambda^{1/4}\ge\pi\sqrt{2}=4.4428\dots\,,
$$
it follows that
$$
G(\lambda)<G_1(\lambda):=
\frac1\pi\left(\frac{\pi\sqrt{2}}{4\lambda^{3/4}}\,
\left(1+4.1e^{-\pi\sqrt{2}\lambda^{1/4}}\right)-
\frac1{2\lambda}\right)
$$
By Theorem~\ref{Th.var}, for $D\ge1$,
substituting $\lambda=3D-3/2$ we obtain
\begin{multline}\label{tor.var.1.2}
\Bbb
V(D)=\min_{\lambda\ge-1}\{(\lambda+D)G(\lambda)\}\le\\\le
\{(\lambda+D)G(\lambda)\}\vert_{\lambda=3(D-1/2)}=(4D-3/2)G(3D-3/2)<
(4D-3/2)G_1(3D-3/2).
\end{multline}
Therefore it suffices to show that
$$
R(D):=(4D-3/2)G_1(3D-3/2)-\frac{\sqrt{2}}{\sqrt[4]{27}}D^{1/4}+\frac2{3\pi}
\le0.
$$
Setting $x:=(3D-3/2)^{1/4}$, $x\ge (3/2)^{1/4}=1.1066\dots$
we have
$$
\aligned
6\pi
R(D)=(8x^4+3)\left(\frac{\pi\sqrt{2}}{4x^3}
\left(1+4.1e^{-\pi\sqrt{2}x}\right)-\frac1{2x^4}\right)
-2\sqrt{2}\pi(x^4+3/2)^{1/4}+4=\\=
2\pi\sqrt{2}\left(\left(x+\frac3{8x^3}-(x^4+3/2)^{1/4}\right)
+\left(x+\frac3{8x^3}\right)4.1e^{-\pi\sqrt{2}x}\right)-\frac3{2x^4}
=:R(x).
\endaligned
$$
The first term in parenthesis is $O(x^{-7})$, hence $R(x)<0$
already for, say,  $x\ge2$. Instead of analyzing $R(x)$ near $x=(3/2)^{1/4}$
we show in Fig.~\ref{fig:R(x)} the graph  of $R(x)$ in this region, so that
$R(x)<0$ for all $x$.
\begin{figure}[htb]
\centerline{\psfig{file=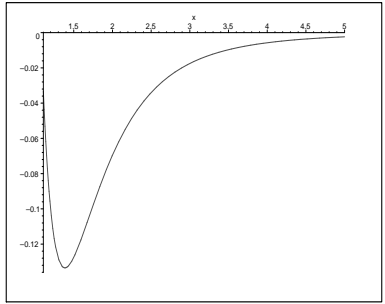,width=4cm,height=5cm,angle=270}}
\caption{Graph of the function $R(x)$ on $x\in[(3/2)^{1/4},5]$.}
\label{fig:R(x)}
\end{figure}

In conclusion we point out that
the general formula~\eqref{lambda(D)} gives
$
\lambda(D)=3D-\frac{4\sqrt{2}}{\pi\sqrt[4]{3}}D^{3/4}+\dots\,
$, which explains  the choice of at least  the leading term
in the substitution $\lambda=3D-3/2$ that has been used above.
 The proof of the
theorem is complete.
\end{proof}

\paragraph{The case $l>0$} The analysis of the Green's function
\eqref{tor.green}  for general $l>0$ is more delicate. Indeed, in this
case, the function $k\to\frac1{|k|^{2m}+\lambda|k|^{2l}}$ has a
singularity at $k=0$. By this reason, the associated fundamental
solution $\Phi_\lambda(x)$ is not rapidly decaying and the Poisson
formula \eqref{tor.poi} becomes not essentially helpful. To overcome
this difficulty, following \cite{Zelik}, we introduce $\mu=\lambda^{-1}$
and rewrite \eqref{tor.poi} as follows
\begin{equation}\label{tor.poi1}
G(\lambda)=(2\pi)^{-n}\mu\sum_{k\in\mathbb{Z}^n_0}
\frac1{k^{2l}(\mu|k|^{2(m-l)}+1)}=\mu\tilde G(\mu).
\end{equation}
Then, differentiating the function $\tilde G(\mu)$ $s$-times in $\mu$,
we get
\begin{equation}\label{tor.poi2}
\frac{d^s\tilde G(\mu)}{d\mu^s}=(-1)^s\Gamma(s+1)(2\pi)^{-n}
\sum_{k\in\mathbb{Z}^n_0}\frac{|k|^{2(m-l)s-2l}}{(\mu|k|^{2(m-l)}+1)^{s+1}}
\end{equation}
and the function
$k\to \frac{|k|^{2(m-l)s-l}}{(\mu|k|^{2(m-l)}+1)^{s+1}}$ becomes
regular at $k=0$ if $s\ge\frac l{m-l}$ (and the rate of decay of this
function as $k\to\infty$ remains $|k|^{2m}$ for all $s$). Thus, we may
apply the Poisson summation formula in order to find the asymptotic
behavior of \eqref{tor.poi2} as $\mu\to0+$
(analogously to \eqref{tor.poi}) which after the $s$-times integration
will give us the asymptotics for the initial function $G(\lambda)$
{\it up to $s$ integration constants} which should be determined using
the alternative methods, see below.
\par
We illustrate this method for the particular case $l=1$ and $m=2$ only.
We also restrict ourselves to considering only the two ($n=2$) and
three ($n=3$) dimensional cases although the general case can be
analyzed in a similar way.
\par
Let $n=2$. Then, we need to differentiate $\tilde G(\mu)$ only once
in order to remove the singularity at $k=0$:
\begin{equation}\label{tor.bg0}
\frac{d \tilde G(\mu)}{d\mu}=-\frac1{4\pi^2}\sum_{k\in\mathbb{Z}^2_0}
\frac1{(1+\mu|k|^2)^2}.
\end{equation}
Applying the Poisson summation formula to this series, we get
\begin{equation}\label{tor.bg1}
\frac{d \tilde G(\mu)}{d\mu}=-\frac1{4\pi^2}
\(\frac\pi\mu-1+2\pi^2\mu^{-3/2}\sum_{k\in\mathbb{Z}^2_0}
|k| K_1(2\pi\mu^{-1/2}|k|)\),
\end{equation}
where $K_n(z)$ is the modified Bessel function of order $n$,
see \cite{Wat} for the details. Integrating these series in $\mu$ and
using the standard formulas for the integrals of Bessel functions,
namely, $K_0'(x)=-K_1(x)$,
we end up with the desired formula
\begin{equation}\label{tor.bg00}
\tilde G(\mu)=\frac1{4\pi^2}\(\pi\log\frac1\mu+\beta+\mu-8\pi
\sum_{k\in\mathbb{Z}^2_0}K_0(2\pi\mu^{-1/2}|k|)\),
\end{equation}
where $\beta$ is an integration constant. As shown in \cite{Zelik}, based on the Hardy lattice formula
$$
\sum_{k\in\mathbb{Z}^2_0}\frac1{|k|^{2s}}=4\zeta(1+s)\beta(1+s),
$$
where $\zeta(s)$ and $\beta(s)$ are the Riemann
zeta and Dirichlet beta functions respectively,
\begin{equation}\label{tor.beta}
 \beta=\pi\gamma+4\beta'(1)=\pi\(2\gamma+2\log2+3\log\pi-4\log\Gamma(1/4)\),
\end{equation}
where $\gamma$ and $\Gamma(z)$ are the Euler-Mascheroni constant and Euler
gamma function respectively. Thus, we have proved the following result.
\begin{lemma} Let $n=m=2$ and $l=1$. Then, at least for $\lambda>0$,
the Green's function \eqref{tor.poi1} can be written as follows:
\begin{multline}\label{tor.bg2}
G(\lambda)=\frac1{4\pi}\lambda^{-1}\log\lambda+
\frac{\beta}{4\pi^2\lambda}+\frac1{4\pi^2\lambda^2}-
\frac1{2\pi}\lambda^{-1}\sum_{k\in\mathbb{Z}^2_0}
K_0(2\pi\lambda^{1/2}|k|)\le\\\le\frac1{4\pi^2\lambda}
\(\pi\log\lambda+\beta+\frac1\lambda\),
\end{multline}
where $\beta$ is defined in \eqref{tor.beta}. Moreover, the term
containing the sum of Bessel functions in the middle part of
\eqref{tor.bg2} is exponentially small
(of order $O(e^{-\pi\lambda^{1/2}})$ as $\lambda\to\infty$.
\end{lemma}
Note that the choice $n=m=2$ and $l=1$ corresponds to the limit case in
\eqref{2.int} where the extra logarithmic term appears and which is
usually called Brezis-Gallouet inequality (see \cite{BG}). Although the
condition \eqref{2.dim} is formally violated, most part of the theory
developed above works in that limit case as well. In particular, as
shown in \cite{Zelik} based on \eqref{tor.bg2}, there exists a constant
$L>0$ such that
\begin{equation}\label{tor.bg3}
\Bbb V(D)\le\frac1{4\pi}\(\log D+\log(1+\log D)+L\),\ \ D\ge1,
\end{equation}
where the constant $\frac1{4\pi}$ is sharp
and the numerically found value of $L=2.15627$. This leads to the
improved Brezis-Galloet inequality in the form
\begin{equation}\label{tor.bg4}
\|u\|^2_{C(\Bbb T^2)}\le
\frac1{4\pi}\|\nabla u\|^2_{L^2(\Bbb T^2)}
\(\log\frac{\|\Delta u\|^2_{L^2(\Bbb T^2)}}{\|\nabla u\|^2_{L^2(\Bbb T^2)}}+
\log\(1+\log\frac{\|\Delta u\|^2_{L^2(\Bbb T^2)}}{\|\nabla u\|^2_{L^2(\Bbb T^2)}}\)+L\).
\end{equation}
Let us now consider the 3D case $n=3$ with $m=2$ and $l=1$. Then,
\begin{equation}\label{tor.3D1}
\tilde G(\mu)=\frac1{8\pi^3}\sum_{k\in\mathbb{Z}^3_0}
\frac1{|k|^2(1+\mu|k|^2)},\ \frac {d\tilde G(\mu)}{d\mu}=
-\frac1{8\pi^3}\sum_{k\in\mathbb{Z}^3_0}\frac1{(1+\mu|k|^2)^2}
\end{equation}
and, applying the Poisson summation formula~\eqref{Poisson} to the second sum, after the
computation of the Fourier transform,
namely, using (see~\cite{S-W})
\begin{equation}\label{Four_Stein}
\mathcal{F}\bigl(1/(1+x^2)^{2}\bigr)(\xi)=
\frac{\pi^2}{(2\pi)^{3/2}}e^{-|\xi|},
\end{equation}
 we end up with
\begin{equation}\label{tor.3D2}
\frac {d\tilde G(\mu)}{d\mu}=\frac1{8\pi^3}\(1-\frac{\pi^2}{\mu^{3/2}}\sum_{k\in\Bbb Z^3}e^{-2\pi\mu^{-1/2}|k|^{1/2}}\),
\end{equation}
and integrating this series in $\mu$, we arrive at
\begin{equation}\label{tor.3D3}
\tilde G(\mu)=\frac1{8\pi^3}\(2\pi^2\mu^{-1/2}+\beta_3+\mu-\pi
\sum_{k\in\mathbb{Z}^3_0}|k|^{-1/2}e^{-2\pi\mu^{-1/2}|k|^{1/2}}\),
\end{equation}
where $\beta_3$ is an integration constant which can be found numerically:
$$
\beta_3=-8.91363291758515127,
$$
see the next section for more details on how to compute it. Thus, we have proved the following lemma.
\begin{lemma} Let $n=3$, $m=2$ and $l=1$. Then, at least for $\lambda>0$,
the Green's function \eqref{tor.poi1} can be written as follows:
\begin{multline}\label{tor.3D4}
G(\lambda)=\frac1{8\pi^3\lambda}\(2\pi^2\lambda^{1/2}+\beta_3+\lambda^{-1}
-\pi\sum_{k\in\mathbb{Z}^3_0}
\frac{e^{-2\pi\lambda^{1/2}|k|^{1/2}}}{|k|^{1/2}}\)\le\\\le
\frac1{8\pi^3\lambda}\(2\pi^2\lambda^{1/2}+\beta_3+\lambda^{-1}\),
\end{multline}
where $\beta$ is defined in \eqref{tor.beta}. Moreover, the term
containing the sum  in the middle part of \eqref{tor.3D4} is
exponentially small (of order $O(e^{-\pi\lambda^{1/2}})$ as
$\lambda\to\infty$).
\end{lemma}
Thus, the Green's function $G(\lambda)$ satisfies the assumptions of
Proposition \ref{Prop6.as} with $\theta=1/2$ and
$g_1,g_2,g_3=(8\pi^3)^{-1}(2\pi^2,\beta_3,1)$, respectively.
Therefore,
\begin{equation}\label{tor.3D5}
\Bbb V(D)=\frac1{8\pi^3}
\(4\pi^2D^{1/2}+2\beta_3-\frac{\beta_3^2-4\pi^2}{2\pi^2}D^{-1/2}\)+
O(D^{-1}),
\end{equation}
where the third term is negative, which suggests the following inequality:
\begin{equation}\label{tor.3D6}
\|u\|_{\infty}^2\le \frac1{2\pi}\|\nabla u\|_{L_2(\Bbb T^3)}\|\Delta u\|_{L_2(\Bbb T^3)}-\frac{-\beta_3}{4\pi^3}\|\nabla u\|^2_{L_2(\Bbb T^3)},
\end{equation}
where all constants are sharp. However, up to the moment, we have
checked this inequality only for large $D$. The next lemma shows that
it holds for all $D\ge1$.
\begin{lemma}\label{L:3.6} The inequality \eqref{tor.3D6} holds for every
$u\in H^2(\Bbb T^3)$ with zero mean.
\end{lemma}
\begin{proof} We only need to check that the inequality
\begin{equation}\label{tor.3D7}
\Bbb V(D)\le\frac1{8\pi^3}\(4\pi^2D^{1/2}+2\beta_3\)
\end{equation}
holds for all $D\ge1$. To this end, we will again use  the variational
representation \eqref{tor.var} where
we put $\lambda=D-\frac12$ and inequality \eqref{tor.3D4}.
Singling out the term $4\pi^2D^{1/2}+2\beta_3$ this gives
\begin{multline}
8\pi^3\Bbb V(D)\le
\(1+\frac D{D-1/2}\)\(2\pi^2 \sqrt{D-1/2}+\beta_3+\frac1{D-1/2}\)=\\
=4\pi^2D^{1/2}+2\beta_3+
\frac1{2D-1}
\(\pi^2\frac{\sqrt{4D-2}}{\sqrt{4D(4D-2)}+4D-1}+\beta_3+4+\frac2{2D-1}\),
\end{multline}
and we only need to check that the last term in the right-hand side
is always negative. Indeed,
$$
\pi^2\frac{\sqrt{4D-2}}{\sqrt{4D(4D-2)}+4D-1}<
\pi^2\frac{\sqrt{4D-2}}{\sqrt{4D(4D-2)}+4D-2}=
\pi^2\frac{1}{\sqrt{4D}+\sqrt{4D-2}}\le\frac{\pi^2}{2+\sqrt2}
$$
if $D\ge1$ and, analogously, $\frac2{2D-1}\le2$. Then
$$
\frac{\pi^2}{2+\sqrt2}+\beta_3+6=-0.022892<0
$$
and the lemma is proved.
\end{proof}
\paragraph{Computing the integration constants}\mbox{}\\
As we have seen above, in the case $l>0$ the Poisson summation formula
allows to find the asymptotic expansions of the Green's function
$G(\lambda)$ only up to some integration constants and the direct
computation of that constants is a non-trivial task since the series
\eqref{tor.green} converge not sufficiently fast, especially for big
$\lambda$. Thus, it looks reasonable to find better (e.g., exponentially)
convergent series for computing them. In the present section, we give an
explicit formula for the sum \eqref{tor.green} in the particular case
$n=3$, $m=2$, $l=1$ considered above in terms of the integrals of the
so-called Jacobi theta functions and using the known relations for the theta functions, we find the formula for the integration constant $\beta_3$ through the very fast convergent and convenient for computations series.
Note also that, although we restrict ourselves to consider only that
case, the presented method has a general nature and is applicable
for computing other integration constants including the case of
anisotropic tori, etc.

 We consider the Jacobi theta function (see, for instance, \cite{Edw})

 \begin{equation}\label{con.jac}
\theta_3(q)=\sum_{k\in\Bbb Z}q^{k^2}.
 \end{equation}
Then, the following identity  holds which is crucial in what follows:
\begin{equation}\label{con.iid}
 \theta_3(e^{-\pi t})=\frac1{\sqrt{t}}\cdot\theta_3(e^{-\pi t^{-1}}),
\end{equation}
and which follows from the Poisson summation formula~\eqref{Poisson} with
$n=1$, $f(x)=e^{-x^2/2}$, $\widehat{f}(\xi)=e^{-\xi^2/2}$,
and $\mu=\sqrt{\frac t{2\pi}}$.
\par
Using also the obvious relation
$$
\frac1{k(1+\mu k)}=\int_0^\infty (1-e^{-t/\mu})e^{-kt}\,dt,
$$
we transform \eqref{tor.3D1} as follows
\begin{equation*}
8\pi^3\tilde G(\mu)=\int_0^\infty(1-e^{-t/\mu})\sum_{k\in\mathbb{Z}^3_0}
[e^{-t}]^{k_1^2}[e^{-t}]^{k_2^2}[e^{-t}]^{k_3^2}\,dt=\int_0^\infty
(1-e^{-t/\mu}) \([\theta_3(e^{-t})]^3-1\)\,dt.
\end{equation*}
Splitting the interval of integration
$\R_+=[0,1]\cup[1,\infty)$ and using \eqref{con.iid}, we arrive at
\begin{multline}\label{con.huge}
8\pi^3\tilde G(\mu)=\int_1^\infty(1-e^{-t/\mu})
\([\theta_3(e^{-t})]^3-1\)\,dt+\\
+\int_0^1(1-e^{-t/\mu})
\(\pi^{3/2}t^{-3/2}\theta_3(e^{-\pi^2t^{-1}})^3-1\)\,dt\\=
\int_1^\infty(1-e^{-t/\mu})
\([\theta_3(e^{-t})]^3-1\)\,dt+\pi^{3/2}\int_1^\infty
(1-e^{-1/(t\mu)})t^{-1/2}
\(\theta_3(e^{-\pi^2t})^3-1\)\,dt-\\
-\int_0^1(1-e^{-t/\mu})\,dt+\pi^{3/2}\int_0^1t^{-3/2}(1-e^{-t/\mu})\,dt.
\end{multline}
Now it is not difficult to find the asymptotic expansions for
$\tilde G(\mu)$ as $\mu\to0+$. Indeed, as elementary calculations show,
\begin{multline}
\int_0^1t^{-3/2}(1-e^{-t/\mu})\,dt=\mu^{-1/2}
\left(\int_0^\infty t^{-3/2}(1-e^{-t})\,dt-
\int_{1/\mu}^\infty t^{-3/2}(1-e^{-t})\,dt\right)=\\=
\mu^{-1/2}(2\sqrt\pi-\int_{1/\mu}^\infty t^{-3/2}\,dt+O(e^{-1/\mu}))=2\sqrt\pi\mu^{-1/2}-2+o_\mu(1).
\end{multline}
Using now that $\theta_3(x)-1=O(x)$ as $x\to0$ we can pass to the limit
$\mu\to0$ in the
integrals in~\eqref{con.huge} containing $\theta_3$
and comparing with  \eqref{tor.3D3}, we see that
\begin{equation}\label{con.beta}
\beta_3=-1-2\pi^{3/2}+
\int_1^\infty\(\theta_3(e^{-t})^3-1\)\,dt+\pi^{3/2}\int_1^\infty
t^{-1/2}\(\theta_3(e^{-\pi^2t})^3-1\)\,dt.
\end{equation}
We do not know whether or not the integral in \eqref{con.beta} can be
computed in closed from, however, it is  convenient for high precision
numerical computation of the constant $\beta_3$. Indeed, expanding back
the Jacobi function in Taylor series and using that
$$
\int_1^\infty t^{-1/2}e^{-k t}\,dt=\sqrt\pi\frac{\operatorname{erfc}(\sqrt k)}{\sqrt k},
$$
where $\operatorname{erfc}(x):=\frac2{\sqrt\pi}\int_x^\infty e^{-t^2}\,dt$, we end up with
\begin{equation}\label{con.fin}
\beta_3=-1-2\pi^{3/2}+
\sum_{k\in\mathbb{Z}^3_0}\(\frac{e^{-|k|^2}}{|k|^2}+
{\pi}\frac{\operatorname{erfc}(\pi|k|)}{|k|}\).
\end{equation}
 We see that the rate of convergence of the series is super-exponential and using  Maple to compute it, we get the desired value
$$
\beta_3=-8.91363291758515127
$$
which has been used in the previous section.


\subsubsection{Inequalities on  spheres}
We  recall the basic facts concerning the spectrum of
the Laplace-Beltrami
operator on the $(d-1)$-dimensional sphere
$\mathbb{S}^{d-1}$:
$$
-\Delta Y_n^k=\Lambda_n Y_n^k,\quad
k=1,\dots,k_d(n),\quad n=1,2,\dots.
$$
Here the $Y_n^k$ are the orthonormal spherical harmonics.
Each eigenvalue
$$
\Lambda_n=n(n+d-2)
$$
has multiplicity
$$
k_d(n)=\frac{2n+d-2}{n}\binom{n+d-3}{n-1}.
$$
In particular, for $d=3,4$ we have
\begin{equation}\label{eig_mult}
\aligned
&\mathbb{S}^2:\ \Lambda_n=n(n+1),\ k_3(n)=2n+1,\\
&\mathbb{S}^3:\ \Lambda_n=n(n+2),\ k_4(n)=(n+1)^2.
\endaligned
\end{equation}
The following identity is essential \cite{S-W}:
for any $\xi\in\mathbb{S}^{d-1}$
\begin{equation}\label{identity}
\sum_{l=1}^{k_d(n)}Y_n^l(\xi)^2=\frac{k_d(n)}{\sigma(d)},
\end{equation}
where $\sigma(d)=2\pi^{d/2}/\Gamma(d/2)$ is the surface area of $S^{d-1}$.

Finally, since the kernel of $\Delta$ is the
one-dimensional subspace of constants,  throughout
below we assume orthogonality to constants:
$$
\bar H^s(\mathbb{S}^{d-1})=
\{\varphi\in  H^s(\mathbb{S}^{d-1}),\ (\varphi,1)=0\}.
$$

\paragraph{ Inequalities on  spheres: $\mathbb{S}^2$}\mbox{}
We consider below  applications of the general theory to
inequalities on the 2D sphere with $A=(-\Delta)^{m}$,
$m>1/2$, and
$B=I$, where $\Delta$ is the Laplace--Beltrami operator.
The first positive eigenvalue of $-\Delta$ is $2$ and
in accordance with~\eqref{3.lambda} we consider the
operator
$$
\mathbb{A}(\lambda)=(-\Delta)^{m}+\lambda\,I,\ \
\lambda>-2^m.
$$
Its Green's function $G_\lambda(x,\xi)$ is
$$
G_\lambda(x,\xi)=\sum_{n=1}^\infty\sum_{k=1}^{2n+1}
\frac{Y_n^k(\xi)Y_n^k(x)}{(n(n+1))^{m}+\lambda},
\qquad x,\xi\in\mathbb{S}^2,
$$
and thanks to~\eqref{identity} the function
$G_\lambda(\xi,\xi)$ is independent of $\xi$ and is given
by
\begin{equation}\label{gS2}
G(\lambda)=
G_\lambda(\xi,\xi)=
\sum_{n=1}^\infty\sum_{k=1}^{2n+1}
\frac{Y_n^k(\xi)^2}{(n(n+1))^{m}+\lambda}=\frac1{4\pi}
\sum_{n=1}^\infty
\frac{2n+1}{(n(n+1))^{m}+\lambda}\,.
\end{equation}
Setting
$
\mu:=\lambda^{-1/m}
$
we have
$$
G(\lambda)=\frac1{4\pi}\mu^m\sum_{n=1}^\infty
(2n+1) \varphi(\mu n(n+1)),
$$
where
$$
\varphi(x)=\frac1{x^m+1}.
$$
Since $\varphi(0)=1$ and $\varphi'(0)=0$, Lemma~\ref{L:E-M}
below gives
\begin{equation}\label{gS2as}
G(\lambda)=\frac1{4\pi}
\left[
\lambda^{-\theta}K-\frac23\lambda^{-1}+
0\cdot\lambda^{-(2-\theta)}
\right]+O(\lambda^{-(3-2\theta)}),
\end{equation}
where $\theta=(m-1)/m$ and
$$
K=\int_0^\infty \varphi(x)dx=\frac{(1-\theta)\pi}{\sin\theta\pi}.
$$
In other words, we have shown that for
$f_\xi(\lambda)=G(\lambda)$
the asymptotic expansion \eqref{6.expand} holds with
the following
$g_1$, $g_2$, $g_3$ independent of $\xi\in \mathbb{S}^2$:
\begin{equation}\label{g1g2g3}
g_1=\frac K{4\pi}=\frac{1-\theta}{4\sin\theta\pi},
\qquad g_2=-\frac1{6\pi},
\qquad g_3=0.
\end{equation}

Next, differentiating~\eqref{gS2} and again using
Lemma~\ref{L:E-M} we obtain the asymptotic expansion for
$g_\xi(\lambda)=g(\lambda)$:
\begin{multline*}
g(\lambda)=-f'(\lambda)=-G'(\lambda)=\frac1{4\pi}
\sum_{n=1}^\infty
\frac{2n+1}{((n(n+1))^{m}+\lambda)^2}=
\frac1{4\pi}\mu^{2m}
\sum_{n=1}^\infty
\frac{2n+1}{((\mu n(n+1))^{m}+1)^2}=\\
\frac1{4\pi}\left[
\lambda^{-(1-\theta)}\theta K-\frac23\lambda^{-2}
+0\cdot\lambda^{-(3-\theta)}
\right]+O(\lambda^{-(4-\theta)}),
\end{multline*}
where we used
$$
\int_0^\infty \varphi(x)^2dx=
\int_0^\infty \frac{dx}{(x^m+1)^2}=\theta K.
$$
This justifies differentiation of
 the asymptotic formula for $f(\lambda)$, as required in
Proposition~\ref{Prop6.as}. Applying it we obtain as a result
 the asymptotics of the function $\mathbb V(D)$ on
$\mathbb{S}^2$.
\begin{theorem}\label{T:S2as}
The function $\mathbb V(D)$ solving
on $\mathbb{S}^2$ the extremal problem
\begin{equation}\label{S2.max}
 \mathbb V(\xi,D):=\sup\bigg\{|u(\xi)|^2:\ \ u\in \bar H^m(\mathbb{S}^2),\ \ \|u\|^2=1,
 \ \ \|(-\Delta)^{m/2}u\|^2=D\bigg\}
 \end{equation}
is independent of $\xi\in\mathbb{S}^2$ and has the
following asymptotic behavior as $D\to\infty$:
\begin{equation}\label{S2.max-as}
 \mathbb V(D)=
 \frac1{4\sin\theta\pi}\left(\frac{1-\theta}\theta\right)^\theta
D^{1-\theta}-\frac1{6\pi\theta}-
\frac{\sin\theta\pi}{18\pi^2}\frac{(1-\theta)^{1-\theta}}{\theta^{3-\theta}}
D^{\theta-1}+O(D^{-(2-2\theta)}),
 \end{equation}
 where $\theta=(m-1)/m$.
\end{theorem}

We now apply the asymptotics of $\mathbb V(D)$ obtained
above to multiplicative  inequalities with remainder terms
on $\mathbb{S}^2$. By Theorem~\ref{Th4.main1} the sharp
constant $K=K_\theta$ in the classical multiplicative
inequality
\begin{equation}\label{classS2}
u(\xi)^2\le\|u\|_\infty^2\le
K_\theta\|u\|^{2\theta}\|(-\Delta)^mu\|^{2(1-\theta)}
\end{equation}
is given by
\begin{equation}\label{K_theta}
K_\theta=\frac1{\theta^\theta(1-\theta)^{1-\theta}}\sup_{\lambda\ge0}
\lambda^\theta G(\lambda)
=
\frac1{\theta^\theta
(1-\theta)^{1-\theta}}\sup_{\lambda\ge0}
\frac1{4\pi}\lambda^\theta\sum_{n=1}^\infty
\frac{2n+1}{(n(n+1))^m+\lambda}.
\end{equation}
We note that~\eqref{K_theta} was obtained in~\cite{I98JLMS}
by a somewhat similar but less general argument than the one used in
Theorem~\ref{Th4.main1}. It was also shown there that for
(integer) $m$, $2\le m\le 7$, we in fact have that
the supremum is attained at infinity
\begin{equation}\label{K_theta_supS2}
\sup_{\lambda\ge0}\lambda^\theta G(\lambda)=
\lim_{\lambda\to\infty}\lambda^\theta G(\lambda)=
\int_0^\infty\frac{dx}{1+x^m},
\end{equation}
which gives
$$
K_\theta=\frac1{4\sin\theta\pi}\left(\frac{1-\theta}\theta\right)^\theta
$$
and, equivalently,
\begin{equation}\label{S2lsec1}
 \mathbb V(D)= \mathbb V_m(D)<
\frac1{4\sin\theta\pi}
\left(\frac{1-\theta}\theta\right)^\theta D^{1-\theta}
\end{equation}
for $2\le m\le 7$. However, for larger $m$'s the supremum
in~\eqref{K_theta_supS2} is attained at a finite point
$\lambda_*<\infty$. An explanation of this phenomenon
for the torus has been given in~\cite{Zelik}; below we
consider the case of the sphere $\mathbb{S}^2$.
\begin{lemma}\label{L:S2-oscillations}
For all sufficiently large $m$ the function
$h(\lambda):=\lambda^\theta G(\lambda)$ attains a global maximum
at a finite point $\lambda_*$ and $h(\lambda_*)>h(\infty)$.

\end{lemma}
\begin{proof}
 Setting in~\eqref{K_theta} $\lambda=\nu^{2m}$ we see
that up to a constant factor, $h(\lambda)$ is equal to
$$
H(\nu)=\nu^{2m-2}\sum_{n=1}^\infty\frac{2n+1}{\nu^{2m}+(n(n+1))^m}\,.
$$

We consider the following partitioning of the half-line
$x\ge 0$ by the points
$$
a_n\,=\,a_n(\nu)\,=\,\frac {(n-1)n}{\nu^2},\ \ n=1,\dots\,.
$$
Then a direct inspection shows that
\begin{equation}\label{H(nu)}
H(\nu)\,=\,
\frac 12 \varphi(a_2)(a_2-a_1)\,+\,\sum\limits_{n=2}^\infty
\frac {\varphi(a_n)+\varphi(a_{n+1})}2\,(a_{n+1}-a_n),
\end{equation}
where $\varphi(x)=1/(1+x^m)$, which looks like a step
function for large $m$: $\varphi(x)\approx 1$ for $x<1$ and
$\varphi(x)\approx 0$ for $x>1$.
In view of Lemma~\ref{L:E-M} below we have
$$
H(\infty):=
\lim_{\nu\to\infty}H(\nu)=
\int_0^\infty\varphi(x)\,dx=\frac{\pi/m}{\sin\pi/m}=1+o_{m\to\infty}(1).
$$
We fix a large $m$, and  set, say,  $\nu=\nu_0=\sqrt{2}+1/100$. Then
$$
\aligned
a_1(\nu_0)=0,\quad a_2(\nu_0)<1 (=0.986),
\quad a_3(\nu_0)\approx 3 (=2.958);\\
a_2(\nu_0)-a_1(\nu_0)=0.986;\quad a_3(\nu_0)-a_2(\nu_0)=1.972.
\endaligned
$$
Since $\varphi(a_2)=1+o_{m\to\infty}(1)$,
$\varphi(a_3)=o_{m\to\infty}(1)$, and the sum from $n=3$ to
$\infty$ in~\eqref{H(nu)} gives a contribution of the order
$o_{m\to\infty}(1)$, it follows that
$$
H(\nu_0)=\frac{0.986}2+\frac{1.972}2+o_{m\to\infty}(1)=
1.479+o_{m\to\infty}(1).
$$
Hence $H(\nu_0)>H(\infty)$, and the proof is complete.
\end{proof}

This argument also explains the initial oscillatory behavior of
$H(\nu)$ when the corresponding node $a_n(\nu)$ hits $1$ with the
increase of $\nu$. Figure~\ref{fig:S2_2-10} shows the monotone
behavior of $H(\nu)$ for $m=2$ and the initial oscillations of $H(\nu)$
for $m=10$. Observe that the first (and global) maximum
in the second case  is located near $\nu=\sqrt{2}$.

\begin{figure}[htb]
\centerline{\psfig{file=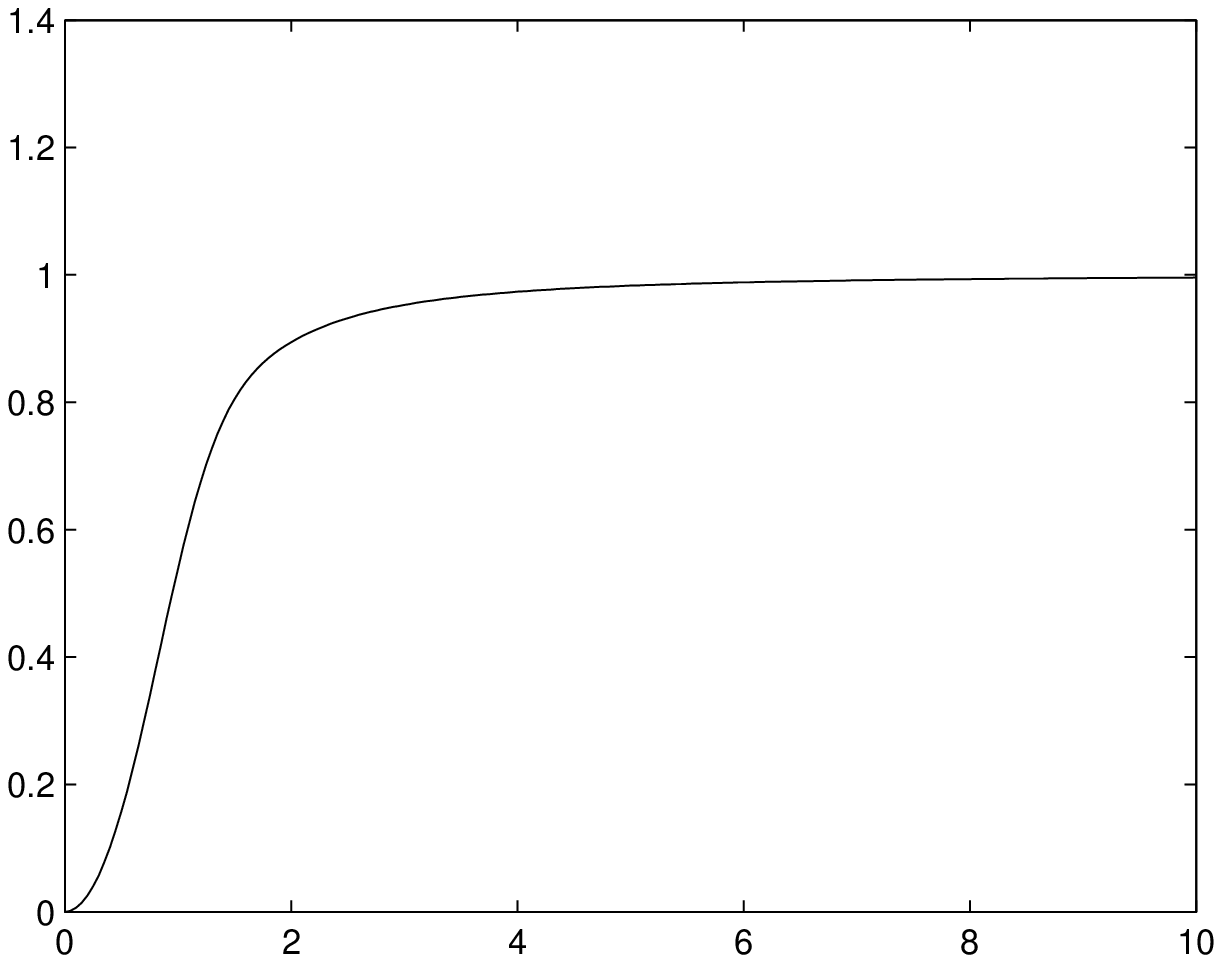,width=8cm,height=6cm,angle=0} 
\psfig{file=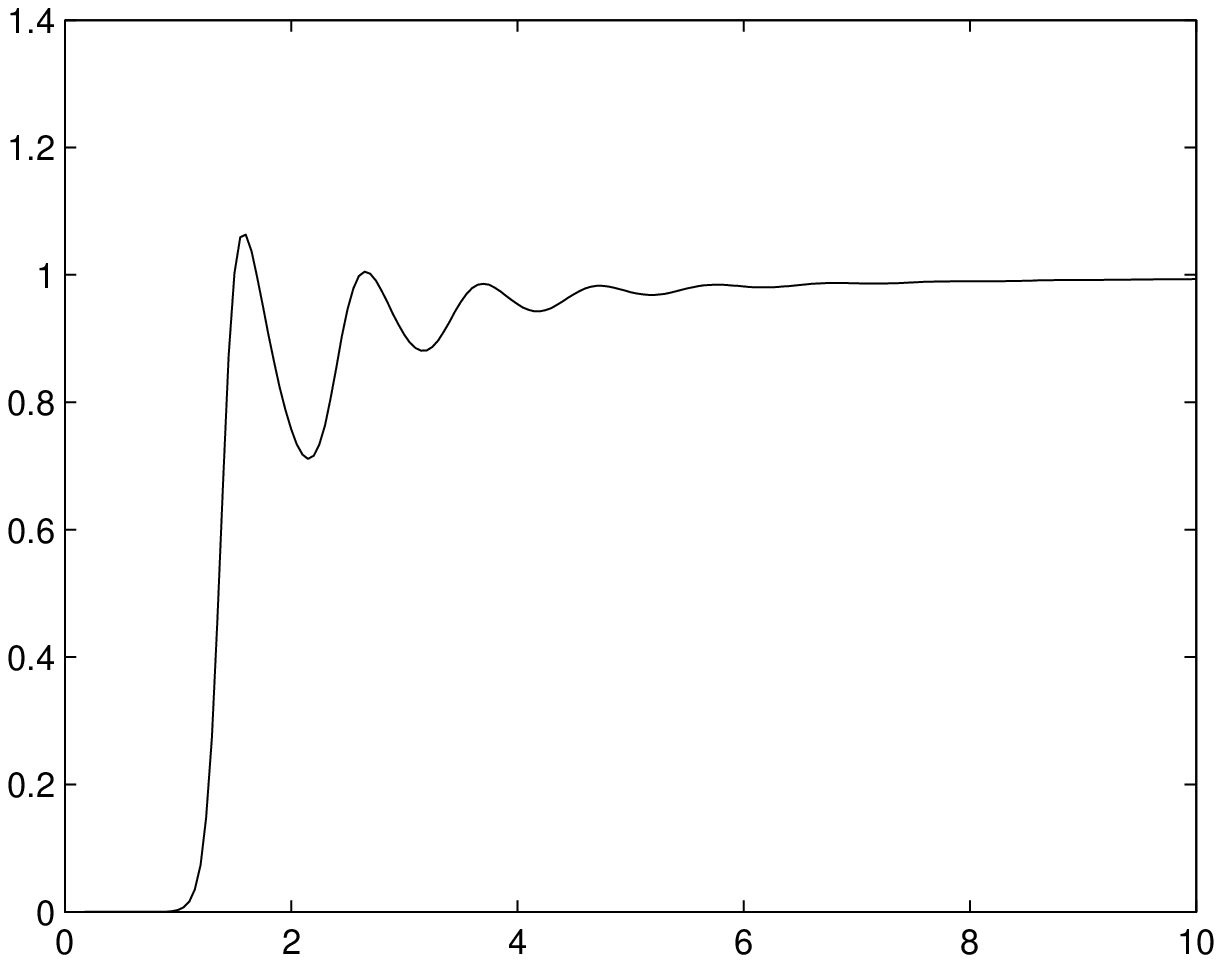,width=8cm,height=6cm,angle=0}} 
\caption{Graphs of the normalized function
$H(\nu)$
\ for $m=2$ and $m=10$}
\label{fig:S2_2-10}
\end{figure}

Turning to the inequalities with remainder terms we note that
the third term in~\eqref{S2.max-as} is negative, therefore the improved
inequality
\begin{equation}\label{S2lsec}
 \mathbb V(D)<
\frac1{4\sin\theta\pi}
\left(\frac{1-\theta}\theta\right)^\theta D^{1-\theta}-
\frac1{6\pi\theta}=:\overline {\mathbb V}(D)
\end{equation}
holds for all $D\ge D_0$, where $D_0$ is
sufficiently large. On the finite interval $[2^m,
D_0]$ computer calculations are reliable;
their results are shown in Fig.\ref{fig:S2l}.
\begin{figure}[htb]
\centerline{\psfig{file=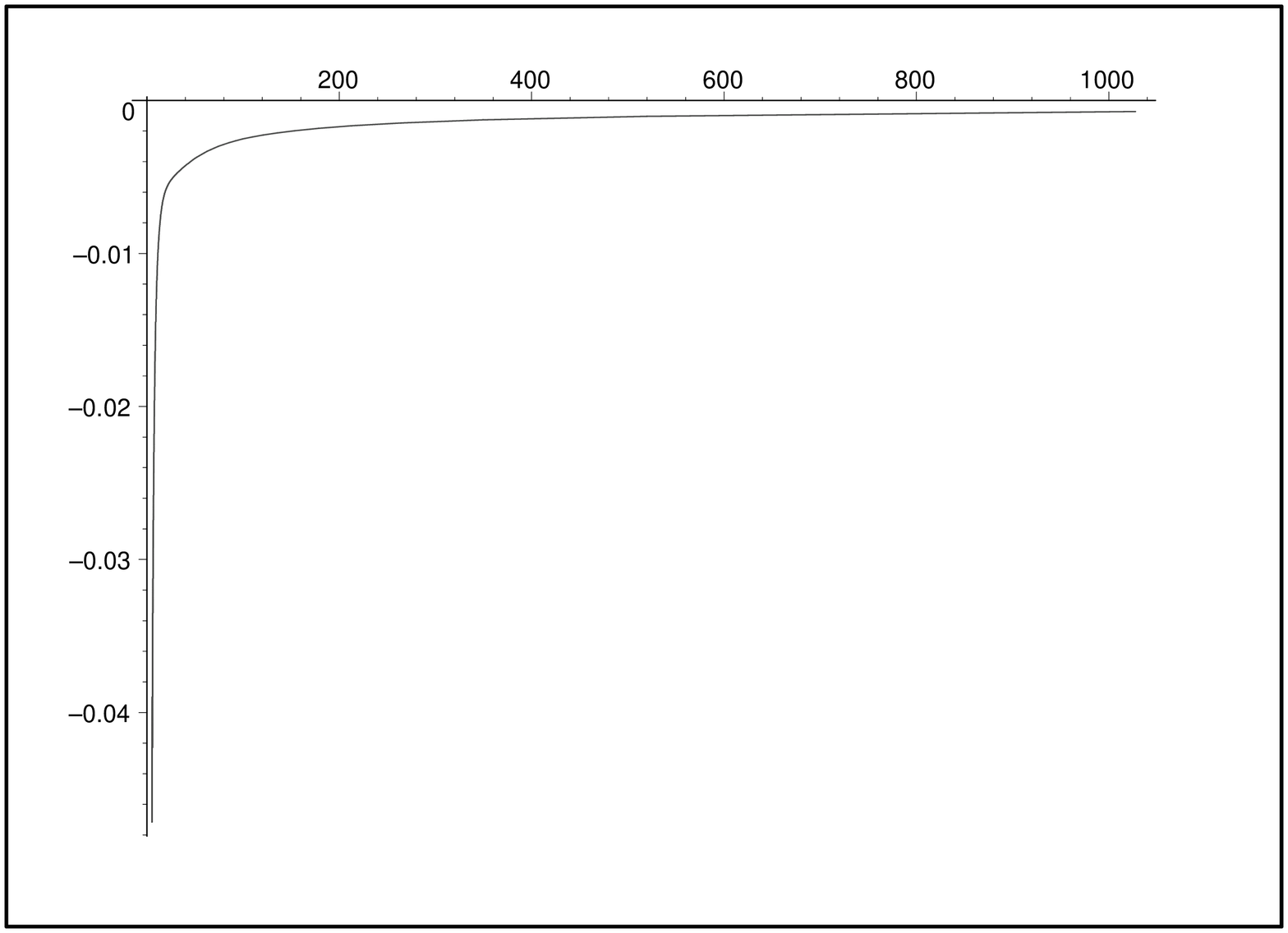,width=6cm,height=5cm,angle=270}
\psfig{file=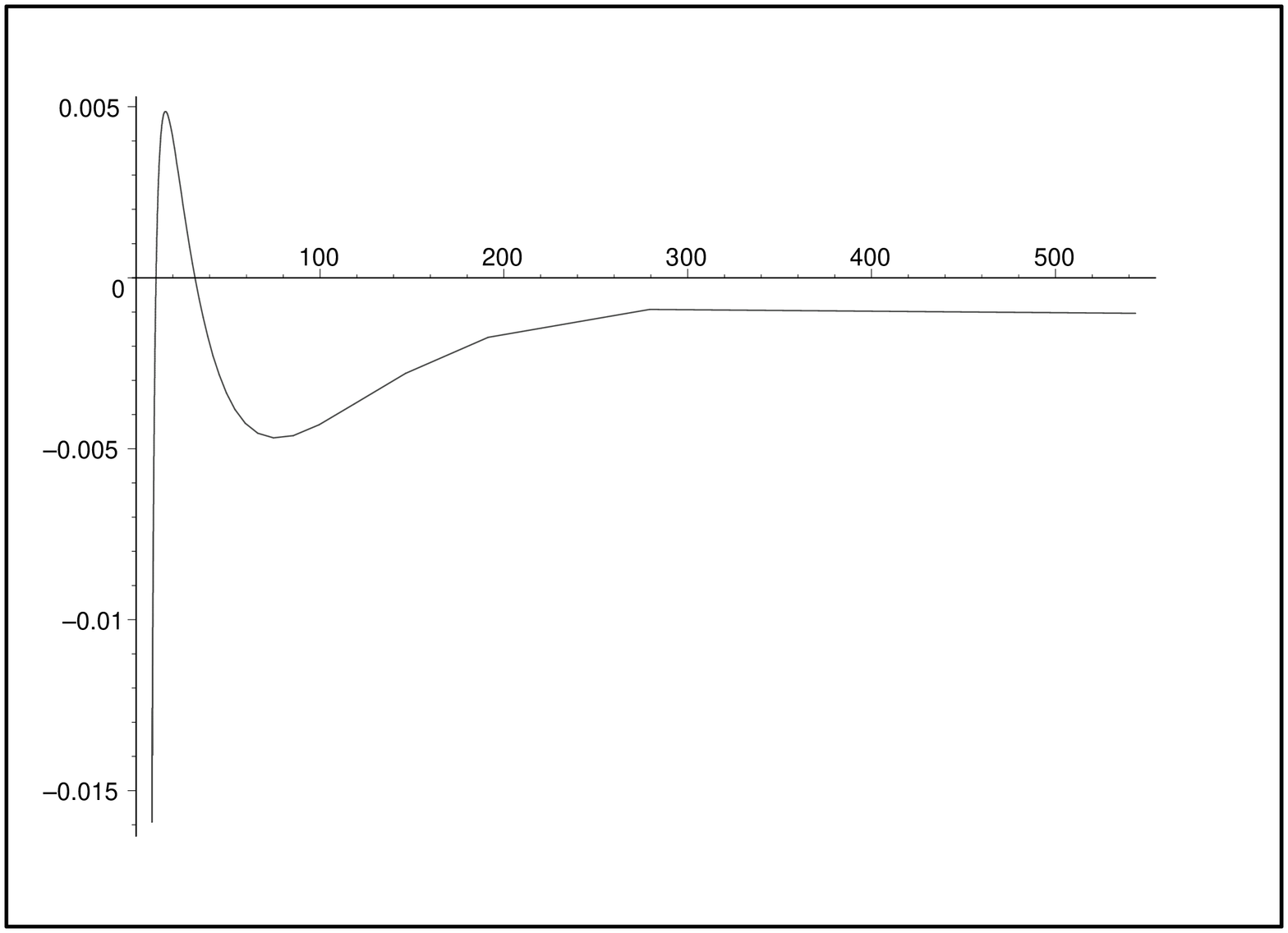,width=6cm,height=5cm,angle=270}
\psfig{file=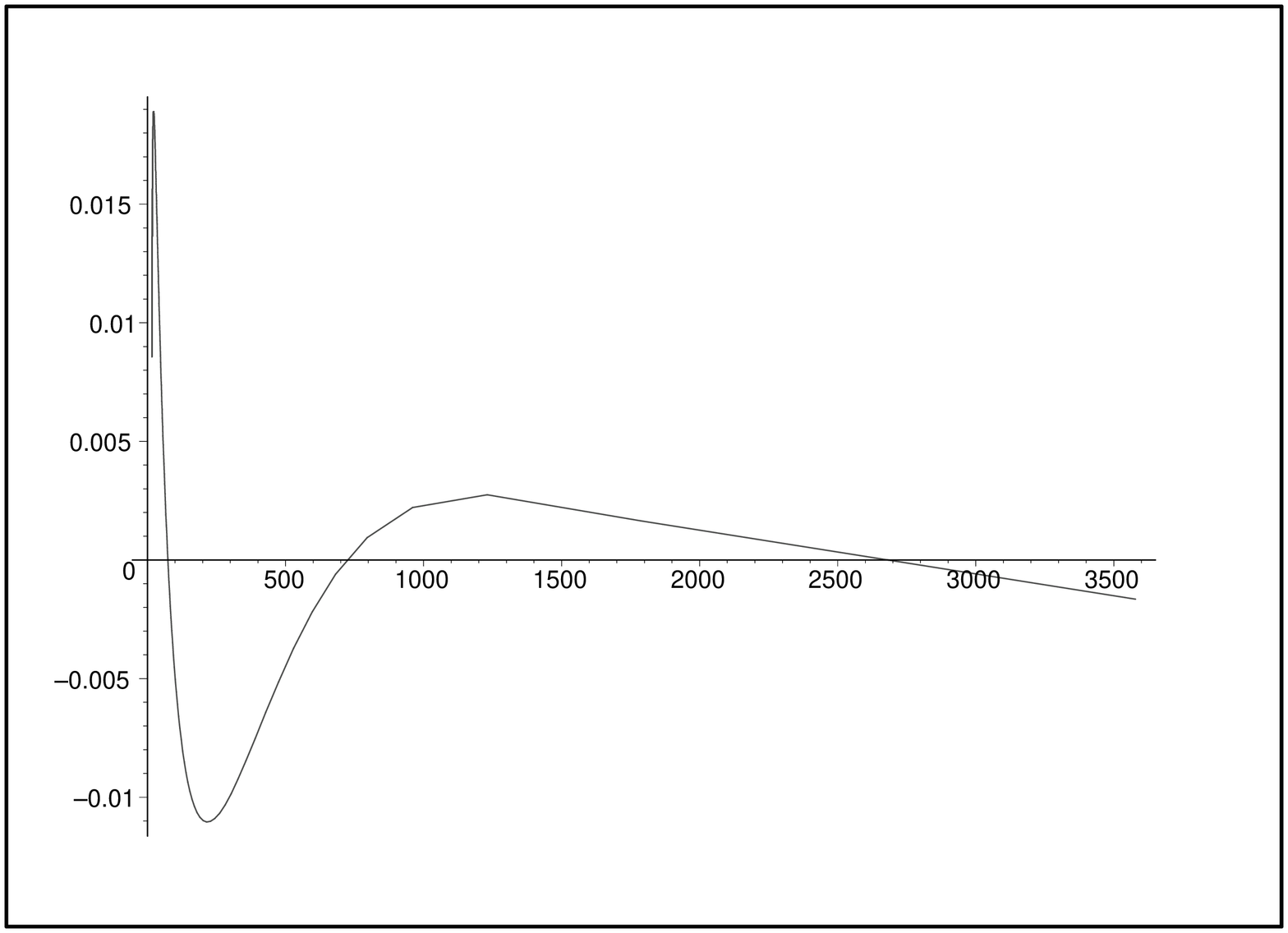,width=6cm,height=5cm,angle=270}
}
\caption{Graph of the function
${\mathbb V}(D)-\overline{\mathbb V}(D)$ for $m=2,3,4$.}
\label{fig:S2l}
\end{figure}

 In
particular, for $m=2$
$$\mathbb V(D)<\overline {\mathbb V}(D)\quad
\text{for all}\quad D,
$$
which proves the following result.
\begin{theorem}\label{T:S22} If $u\in\bar
H^2(\mathbb{S}^{2})$, then
\begin{equation}\label{S22corr}
\|u\|_\infty^2\le
\frac1{4}
\|u\|\|\Delta u\|-\frac1{3\pi}\|u\|^2,
\end{equation}where both constants are sharp and
no extremal functions exist.
\end{theorem}
For $m=3$ and $m=4$ the function
${\mathbb V}_m(\xi,D)-\overline{\mathbb V}_m(\xi,D)$
attains a global maximum $0.00486\dots$
at $D=15.8\dots$  for $m=3$
and a global maximum $0.0189\dots$ at $D=22.4\dots$
for $m=4$, respectively.
Accordingly, we have
\begin{equation}\label{S223and4}
\|u\|_\infty^2\le
\frac1{4\sin\theta\pi}\left(\frac{1-\theta}\theta\right)^\theta
\|u\|^{2\theta}\|(-\Delta)^{m/2}u\|^{2(1-\theta)}-\frac{\varepsilon_m}{6\pi\theta}\|u\|^2,
\end{equation}
where $m=3,4$, $\theta=(m-1)/m$ and $\varepsilon_3=0.938\dots$,
$\varepsilon_4=0.821\dots$.

\medskip
\medskip

We now prove the asymptotic formula
for   functions defined by the
following series~\eqref{F} that
we have been  systematically using above.
Let $F(\mu)$ be defined as follows
\begin{equation}\label{F}
F(\mu):=\sum_{n=1}^\infty(2n+1)f\left(\mu\, n(n+1)\right),
\end{equation}
where $f$ is sufficiently smooth and sufficiently fast
decays at infinity. We need to find the asymptotics of
$F(\mu)$ as $\mu\to0$.
\begin{lemma}\label{L:E-M}
The following asymptotic expansion holds as $\mu\to0$:
\begin{equation}\label{as2}
F(\mu)=\frac1{\mu}\int_0^\infty f(x)dx-\frac23f(0)-
\frac1{15}\mu f'(0)+
O(\mu^2).
\end{equation}
\end{lemma}
\begin{proof}
We set
$$
R(x)=(2x+1)f(\mu\,x(x+1))
$$
and observe that
\begin{equation}\label{Rint}
\int_0^\infty R(x)dx=\frac1\mu\int_0^\infty
f(\mu\,x(x+1))d(\mu\,x(x+1))=\frac1\mu\int_0^\infty
f(x)dx.
\end{equation}
We calculate the derivatives of $R$ up the order 5
at $x=0$:
\begin{equation}\label{derR}
\aligned
&R(0)=f(0);\\
&R'(0)=2f(0)+\mu f'(0);\\
&R''(0)=6\mu f'(0)+\mu^2f''(0);\\
&R'''(0)=12\mu f'(0)+12\mu^2f''(0)+\mu^3f'''(0);\\
&R^{(4)}(0)=60\mu^2f''(0)+20\mu^3f'''(0)+ \mu^4f^{(4)}(0);\\
&R^{(5)}(0)=120\mu^2f''(0)+180\mu^3f'''(0)+30\mu^4f^{(4)}(0)+\mu^5f^{(5)}(0),
\endaligned
\end{equation}
and, in addition,
\begin{multline}\label{R6}
R^{(6)}(x)=
840\mu^3(2x+1)f^{(3)}(\mu x(x+1))+420\mu^4(2x+1)^3f^{(4)}(\mu x(x+1))+\\
+42\mu^5(2x+1)^5f^{(5)}(\mu x(x+1))+\mu^6(2x+1)^7f^{(6)}(\mu
x(x+1)).
\end{multline}

Next we use the Euler--Maclaurin formula~
(see, for instance,~\cite{Krylov})
\begin{equation}\label{E-M}
\sum_{n=0}^\infty R(n)=\int_0^\infty R(x)dx+\frac12R(0)-
\sum_{i=2}^k\frac{B_i}{i!}R^{(i-1)}(0)-\int_0^\infty
\frac{B_k(x)}{k!}R^{(k)}(x)dx,
\end{equation}
where the $B_k$'s are the Bernoulli numbers:
$B_2=\frac16$, $B_3=0$, $B_4=-\frac 1{30}$, $B_5=0$,
$B_6=\frac1{42}$, $\dots$, and the $B_k(x)$'s are the
periodic Bernoulli polynomials. Using~(\ref{E-M}) with
$k=6$ and taking into account~(\ref{Rint}) and (\ref{derR})
we obtain
\begin{equation*}\label{E-M-R}
\aligned
F(\mu)=-R(0)+
\sum_{n=0}^\infty R(n)=\frac1\mu\int_0^\infty f(x)dx-\frac12f(0)-
\frac1{12}(2f(0)+\mu f'(0))+\\
+\frac1{720}(12\mu f'(0)+12\mu^2f''(0)+\mu^3f'''(0))-\\
-\frac{1}{42\cdot 720}
(120\mu^2f''(0)+180\mu^3f'''(0)+30\mu^4f^{(4)}(0)+\mu^5f^{(5)}(0))
-\int_0^\infty
\frac{B_6(x)}{6!}R^{(6)}(x)dx.
\endaligned
\end{equation*}
This gives \eqref{as2} provided that the remainder integral
term is of the order $O(\mu^2)$. The periodic Bernoulli
polynomials are clearly bounded on $(0,\infty)$. Therefore
the contribution of each term in \eqref{R6} is of the order
$\mu^2$. For example, the last term is of the order
$$
\mu^6\int_0^\infty x^7 g(\mu
x^2)dx=\frac12\mu^2\int_0^\infty y^3g(y)dy.
$$
The three remaining terms in~\eqref{R6} are treated
similarly.
\end{proof}

\paragraph{ Inequalities on  spheres: $\mathbb{S}^3$}\mbox{}
We consider on the 3D sphere $\mathbb{S}^3$
only one example with $l=1$ and $m=2$, so that $\theta=1/2$.
We set $A=(-\Delta)^{2}$, $B=-\Delta$, and let
$$
\mathbb{A}(\lambda)=(-\Delta)^{2}-\lambda\Delta.
$$
The Green's function of $\mathbb{A}(\lambda)$ is
$$
G_\lambda(x,\xi)=\sum_{n=1}^\infty\sum_{k=1}^{(n+1)^2}
\frac{Y_n^k(\xi)Y_n^k(x)}{n(n+2)\bigl(n(n+2)+\lambda\bigr)}
$$
and again using~\eqref{identity} we see that
$G_\lambda(\xi,\xi)$ is independent of $\xi$:
\begin{equation}\label{gS3}
G(\lambda)=G_\lambda(\xi,\xi)=
\frac1{2\pi^2}
\sum_{n=1}^\infty
\frac{(n+1)^2}{n(n+2)\bigl(n(n+2)+\lambda\bigr)},
\ \ \text{for any}\ \ \xi\in\mathbb{S}^3.
\end{equation}

We do not need an analogue of Lemma~\ref{L:E-M} here since
the series in~\eqref{gS3} can be summed explicitly \cite{I98JLMS}:
\begin{equation}\label{gS3expl}
G(\lambda)=\frac1{2\pi^2}
\left[\frac{\sqrt{\lambda-1}}\lambda\frac\pi 2
\coth\pi\sqrt{\lambda-1}-\frac{\lambda-1}{\lambda^2}+
\frac1{4\lambda}
\right],
\end{equation}
and, in addition,
\begin{equation}\label{K_theta_sup}
\sup_{\lambda\ge0}\lambda^\theta G(\lambda)=
\lim_{\lambda\to\infty}\lambda^\theta G(\lambda)=
\frac1{4\pi},
\end{equation}
which gives in view of Theorem~\ref{Th4.main1}
the following sharp inequality on $\mathbb{S}^3$ \cite{I98JLMS}
\begin{equation}\label{S312}
\|u\|^2_\infty\le\frac 1{2\pi}\|\nabla u\|\|\Delta u\|.
\end{equation}

To obtain power expansion of $G(\lambda)$ we can
replace $\coth\pi\sqrt{\lambda-1}$ by $1$ which gives :
$$
G(\lambda)=
\frac1{2\pi^2}
\left[\frac\pi 2\lambda^{-1/2}-\frac34\lambda^{-1}-
\frac\pi 4\lambda^{-3/2}+O(\lambda^{-2})
\right],
$$
so that \eqref{6.expand} holds with:
\begin{equation}\label{g1g2g3S3}
g_1=\frac 1{4\pi},
\qquad g_2=-\frac3{8\pi^2},
\qquad g_3=-\frac1{8\pi}.
\end{equation}

Using Proposition~\ref{Prop6.as} we obtain the following result.
\begin{theorem}\label{T:S3as}
The solution $\mathbb V(D)$ of the extremal problem
\begin{equation}\label{S3.max}
 \mathbb V(\xi,D):=\sup\bigg\{|u(\xi)|^2:\ \ u\in \bar H^2(\mathbb{S}^3),
 \ \ \|\nabla u\|^2=1,
 \ \ \|\Delta u\|^2=D\bigg\}
 \end{equation}
is independent of $\xi\in\mathbb{S}^3$ and has the
following asymptotic behavior as $D\to\infty$:
\begin{equation}\label{S3.max-as}
 \mathbb V(D)=
\frac1{2\pi}D^{1/2}-
\frac{3}{4\pi^2}-\left(\frac{9+4\pi^2}{16\pi^3}\right)D^{-1/2}
+O(D^{-1}).
 \end{equation}
\end{theorem}

The third term in~(\ref{S3.max-as}) is negative, therefore
\begin{equation}\label{Thetainf}
 \mathbb V(D)<\frac1{2\pi}\sqrt{D}-
\frac3{4\pi^2}
\end{equation}
for all $D\ge D_0$, where $D_0$ is
sufficiently large. However, similarly to Theorem~\ref{Th.tor.1.1}
and  Lemma~\ref{L:3.6}, taking the advantage of the explicit
formula~\eqref{gS3expl}, we have the following result.
\begin{lemma}\label{L:S3-analyt}
Inequality~\eqref{Thetainf} holds holds for all $D\ge\sqrt{3}$.
\end{lemma}
\begin{proof}
In view of Theorem~\ref{Th.var}, for $D\ge\sqrt{3}$ (the first eigenvalue
of $(-\Delta)^{1/2}$ on $\mathbb{S}^3$  is $\sqrt{3}$)
\begin{equation}\label{S3.var}
\Bbb V(D)=\min_{\lambda\ge-\sqrt{3}}\{(\lambda+D)G(\lambda)\}\le
\{(\lambda+D)G(\lambda)\}\vert_{\lambda=D}=2DG(D).
\end{equation}
Hence, using~\eqref{gS3expl}, the inequality
$\sqrt{D-1}\le\sqrt{D}-1/({2\sqrt{D}})-1/({8D^{3/2}})$,
and replacing $\coth\alpha$ in the negative terms below by
$1$, we obtain
$$
\aligned
\mathbb V(D)-\frac1{2\pi}\sqrt{D}+
\frac3{4\pi^2}\le 2DG(D)-\frac1{2\pi}\sqrt{D}+
\frac3{4\pi^2}=\\
\frac{2D}{2\pi^2}\left(\frac{\sqrt{D-1}}D\frac\pi 2
\coth\pi\sqrt{D-1}-\frac{D-1}{D^2}+
\frac1{4D}   \right)-\frac1{2\pi}\sqrt{D}+
\frac3{4\pi^2}=\\
\frac1{2\pi}\left(\sqrt{D}\left(\coth\pi\sqrt{D-1}-1\right)
 -\frac1{2\sqrt{D}}-\frac1{8D^{3/2}}+\frac2{\pi D}
 \right)=\\=
\frac1{2\pi
D^{3/2}}\left(D^2\left(\coth\pi\sqrt{D-1}-1\right)-
\frac D2-\frac18+\frac{2\sqrt{D}}\pi\right)
=\frac1{2\pi D^{3/2}}\left(F_1(D)-F_2(D)\right)<0,
\endaligned
$$
since $F_1(D)=D^2\left(\coth\pi\sqrt{D-1}-1\right)$
 is decreasing for $D\ge\sqrt{3}$ and
$F_1(D)\le
F_1(\sqrt{3})=3\coth(\pi\sqrt{\sqrt{3}-1}-1)=0.0278\dots$,
while $F_2(D)=D/2+1/8-2\sqrt{D}/\pi$
is increasing and
$F_2(D)\ge F_2(\sqrt{3})=\sqrt{3}/2+1/8-2\sqrt[4]{3}/\pi=0.153\dots$\,.
The proof is complete.
\end{proof}

As an immediate corollary we obtain  inequality~(\ref{S312}) in the
following refined form.

\begin{theorem}\label{S312ex}
If $u\in\bar{H}^2(\mathbb{S}^3)$, then
\begin{equation}\label{12two}
\|u\|^2_\infty\le\frac 1{2\pi}\|\nabla u\|
\|\Delta u\|-\frac3{4\pi^2}\|\nabla u\|^2,
\end{equation}
where both constants are sharp and no extremal functions
exist.
\end{theorem}
\begin{remark}\label{R:leading-term}
{\rm
The coefficient of the leading term in
the asymptotic expansions~\eqref{S2.max-as},
\eqref{S3.max-as} coincides with the corresponding
constant in $\mathbb{R}^n$.
}
\end{remark}

\subsection{Manifolds with boundary} \phantom{eggog}

\subsubsection{First-order inequality on the half-line}\label{SS:Half}
We consider a couple of inequalities on the half-line
$\mathbb{R}_+=(0,\infty)$. On the whole $\mathbb{R}$ we
have
\begin{equation}\label{0lR}
\|u\|_\infty^2\le \|u\|\|u'\|,
\end{equation}
where $u\in H^1(\mathbb{R})$ and where the constant $1$ is
sharp and there exists a unique extremal function
$u_*(x)=e^{-|x|}$. Using the extension by zero we see that
inequality~\eqref{0lR} still holds on $\mathbb{R}_+$, and
there are no extremal functions since $u_*(x)>0$. We can
obtain the following refined form of this inequality
\begin{equation}\label{0-infty}
 u(\xi)^2\le\|u\|\|u'\|(1-e^{-2\xi}),\qquad u\in H^1_0(\mathbb{R}_+)
\end{equation}
from the general theory developed above. In fact, the
Green's function of the operator
$$
\mathbb{A}u=-u''+u, \qquad u(0)=u(\infty)=0
$$
is
$$
G(x,\xi)=\left\{
\begin{array}{ll}
e^{-\xi}\sinh x, & x<\xi,\\
e^{-x}\sinh\xi, & x\ge \xi. \\
\end{array}
\right.
$$
As in Lemma~\ref{Lem3.est} we have
$$
u(\xi)^2\le
 \bigl(\|u'\|^2+\|u\|^2\bigr)G(\xi,\xi)=
\bigl(\|u'\|^2+\|u\|^2\bigr)\frac{1-e^{-2\xi}}2\,,
$$
and \eqref{0-infty} follows by the standard scaling
argument.

\subsubsection{Inequality for Bessel operator on the half-line}\label{SS:Bess}
Here we consider a one-dimensional inequality on the
half-line $\mathbb{R}^+$ with $Au=-u''-u/4x^2$ and $Bu=u$.
Although the corresponding sharp constant was found in
\cite{Lapt_unpubl} in connection with Lieb--Thirring
inequalities for radial potentials (see also \cite{Ekh-Fr})
we include this example to illustrate our general approach.
We first observe that in view of the one-dimensional Hardy
inequality
$$
\int_0^\infty\frac{u(x)^2}{x^2}dx\le 4\int_0^\infty
(u'(x))^2dx,
$$
the operator $A$ is non-negative. In accordance with the
general theory developed in section~\ref{s1} in order to
find the sharp constant $C=C(\xi)$ in the inequality
\begin{equation}\label{Hardy-Bessel}
|u(\xi)|^2\le
C(\xi)\left(\int_0^\infty\left((u'(x))^2-\frac{u(x)^2}{4x^2}\right)
dx\right)^{1/2}\|u\|,
\end{equation}
we need to write down the Green's function of the following
operator
$$
\mathbb{A}(\lambda) u= -u''-u/4x^2+\lambda u
$$
with Dirichlet boundary conditions $u(0)=u(\infty)=0$. We
have
\begin{equation}\label{GrBess}
G_\lambda(x,\xi)=
\left\{
\begin{array}{ll}
\sqrt{x\xi}\,K_0(\sqrt{\lambda} \xi)I_0(\sqrt{\lambda}x),\ x<\xi,& \\
\sqrt{x\xi}\,K_0(\sqrt{\lambda} x)I_0(\sqrt{\lambda}\xi),\ x>\xi,& \\
\end{array}
\right.
\end{equation}
where $K_0$ and $I_0$ are the modified Bessel functions of
zeroth order. In fact, the functions
$\sqrt{x}K_0(\sqrt{\lambda}x)$ and
$\sqrt{x}I_0(\sqrt{\lambda}x)$ both satisfy the homogeneous
equation, $I_0(0)=0$, $K_0(\infty)=0$, and the jump
condition $G(\xi+0,\xi)-G(\xi-0,\xi)=1$
 is satisfied in view of the Wronski identity
$ I_0(x)K_0'(x)-I_0'(x)K_0(x)=-1/x $. Therefore by
Theorem~\ref{Th4.main1}, the constant  $C(\xi)$ is, in
fact, independent of $\xi$ and is given by
\begin{multline*}
C(\xi)=2\sup_{\lambda\ge0}\sqrt{\lambda}G_\lambda(\xi,\xi)
=2\sup_{\lambda\ge0}\sqrt{\lambda}\xi
K_0(\sqrt{\lambda} \xi) I_0(\sqrt{\lambda}\xi)=\\=
2\sup_{r\ge0}r K_0(r) I_0(r)=2r K_0(r)
I_0(r)\vert_{r=r_*=1.075\dots}=2\cdot
0.533\dots=1.06\dots=:C_*.
\end{multline*}
Furthermore, for every fixed $\xi>0$ inequality~\eqref{Hardy-Bessel} turns into equality for $C(\xi)=C_*$ and $u(x)=G_{\lambda_*}(x,\xi)$, where $\lambda_*=\lambda_*(\xi)=r_*^2/\xi^2$. We
also observe that $I_0(x)\sim\sqrt{\frac1{2\pi x}}\,e^x$
and $K_0(x)\sim\sqrt{\frac\pi{2x}}\,e^{-x}$ as
$x\to\infty$, and hence
$$
\lim_{r\to\infty}rK_0(r)I_0(r)=\frac12.
$$

\subsubsection{First-order inequality on
interval}\label{SS:Dir}
In this section we consider the correction term for the
interpolation inequality
\begin{equation}\label{0l}
\|u\|_\infty^2\le 1\cdot\|u\|\|u'\|, \qquad u\in H^1_0(0,L)
\end{equation}
on an interval $(0,L)$ with zero boundary conditions. The
constant $1$ is sharp, since it is sharp for the
inequality~\eqref{0lR} on $\mathbb{R}$
 and we can use extension by
zero. The unique extremal function in~\eqref{0lR} does not
vanish, therefore there are no extremal functions
in~\eqref{0l}, and  one might expect that an inequality
similar to~\eqref{tor.int} holds in the case of a finite
interval:
\begin{equation}\label{dir}
\|u\|_\infty^2\le \|u\|\|u'\|-c_L\|u\|^2, \quad u\in
H^1_0(0,L).
\end{equation}

In fact, $c_L=0$ as we now show. First, by scaling,
if~\eqref{dir} holds, then $c_L=c/L$ where $c$ is an
absolute constant. Next, we consider the truncated
extremal function
$$
\varphi_a(x)=
\left\{
\begin{array}{ll}
    e^{-x}, & \hbox{$0\le x\le a$;} \\
    e^{-a}(a+1-x), & \hbox{$a\le x\le a+1$;} \\
    0, & \hbox{$x\ge0$,} \\
\end{array}
\right.
$$
and set $\varphi_a(-x)=\varphi_a(x)$. Then
$$
\|\varphi_a\|_\infty=1,\quad
\|\varphi_a\|^2=1-\frac13e^{-2a},\quad
\|\varphi_a'\|^2=1+e^{-2a}.
$$
Substituting this into~\eqref{dir} and letting $a\to\infty$
we obtain that~\eqref{dir} can only hold with $c=0$ as claimed.

However, the correction term exists but is exponentially small.
More precisely, we have the following result.
\begin{theorem}\label{T:exp}
Let $u\in H^1_0(0,L)$. Then
\begin{equation}\label{dir-exp}
\|u\|^2_\infty\le\|u\|\|u'\|\bigl(1-2e^{-\frac{L\|u'\|}{\|u\|}}\bigr).
\end{equation}
The  coefficients of the two terms on the right-hand side
are sharp and no extremal functions exist.
\end{theorem}
\begin{proof}
Without loss of generality we set $L=1$. The Green's
function of the  boundary value problem
$$
-y''+\lambda y=\delta(x,\xi), \qquad y(0)=y(1)=0
$$
is
\begin{equation}\label{Green10}
G_\lambda(x,\xi)=\frac1{\lambda^{1/2}\sinh\lambda^{1/2}}
\left\{
\begin{array}{ll}
\sinh\lambda^{1/2}x\sinh\lambda^{1/2}(1-\xi), & 0\le x\le\xi; \\
\sinh\lambda^{1/2}\xi\sinh\lambda^{1/2}(1-x), & \xi\le x\le1. \\
\end{array}
\right.
\end{equation}
In addition, $G_\lambda(\xi,\xi)$ attains its maximum value
with respect to $\xi$ at $\xi=1/2$:
$$
\aligned
 G_\lambda(\xi,\xi)
\le
\frac1{2\lambda^{1/2}\sinh\lambda^{1/2}}
\left(\cosh\lambda^{1/2}-1\right)=
\frac1{2\lambda^{1/2}}\tanh\frac{\lambda^{1/2}}2=
G_\lambda(1/2,1/2).
\endaligned
$$
As in Lemma~\ref{Lem3.est}
$$
u(\xi)^2\le G_\lambda(\xi,\xi)(\|u\|^2+\lambda\|u'\|^2) \le
G_\lambda(1/2,1/2)(\|u\|^2+\lambda\|u'\|^2),
$$
where we have the equality at the point $1/2$ for
$u(x)=\mathrm{const}\,G_\lambda(x,1/2)$. Hence,
$$\mathbb V(\xi,D)\le \mathbb V(1/2,D),$$
and as in~\eqref{4.fgh},\eqref{4.fgh1} we obtain the
parametric representation of $\mathbb V(D):=\mathbb
V(1/2,D)$:
\begin{equation}\label{Dir.fgh1}
D(\lambda)=\frac{h(\lambda)}{g(\lambda)},\qquad \mathbb
\mathbb V(D(\lambda))=
\frac{f(\lambda)^2}{g(\lambda)},\qquad
\lambda\in[-\pi^2,\infty),
\end{equation}
where
\begin{equation}\label{Dir.fgh}
f(\lambda)=
\frac1{2\lambda^{1/2}}\tanh\frac{\lambda^{1/2}}2,\qquad
g(\lambda)=-f'(\lambda),\qquad
h(\lambda)=f(\lambda)-\lambda g(\lambda),
\end{equation}
and~\eqref{dir-exp} is equivalent to the inequality
\begin{equation}\label{forallD}
\mathbb V(D)\le\sqrt{D}(1-2e^{-\sqrt{D}})\quad
\text{for}\quad D\ge\pi^2.
\end{equation}

Using the variational representation from
Theorem~\ref{Th.var} we have
\begin{equation}\label{varrep}
\mathbb V(D)=\min_{\lambda\in(-\pi^2,\infty)}
\frac1{2\sqrt{\lambda}}\tanh\frac{\lambda^{1/2}}2\cdot(\lambda+D).
\end{equation}
To simplify notation we denote $d:=\sqrt{D}$ and further
set
\begin{equation}\label{lam*}
\lambda_*=\lambda_*(D)=D(1-2\sqrt{D}e^{-\sqrt{D}})^2=
d^2(1-2de^{-d})^2.
\end{equation}
As we shall see below  $\lambda_*(D)$ contains the first
terms of the asymptotic expansion as $D\to\infty$ of the
unique solution of the first equation in~\eqref{Dir.fgh1},
or, equivalently, of the unique point $\lambda$ where the global
minimum in~\eqref{varrep} is attained. For the moment, however, we just
substitute
$\lambda_*(D)$ into the right-hand side of \eqref{varrep}
and see that we prove~\eqref{forallD} if we can show that
the following inequality holds for all $d\ge\pi$:
\begin{equation}\label{subslambda}
\frac12\frac1{d(1-2de^{-d})}\tanh\frac{d(1-2de^{-d})}2
\left(d^2(1-2de^{-d})^2+d^2\right)\le
d(1-2e^{-d}),
\end{equation}
or
\begin{equation}\label{subslambda1}
\tanh\frac{d(1-2de^{-d})}2
\left((1-2de^{-d})^2+1\right)\le
2(1-2e^{-d})(1-2de^{-d}).
\end{equation}
Next, we use that $\tanh(u/2)\le 1-2e^{-u}+2e^{-2u}$,
and we also observe that the quadratic polynomial $1-2x+2x^2=
2(x-1/2)^2+1/2$ is monotone decreasing  for $x\in[0,1/2]$
so that
$$
\tanh\frac{d(1-2de^{-d})}2\le
1-2e^{-d(1-2de^{-d})}+2e^{-2d(1-2de^{-d})}\le
1-2e^{-d}(1+2d^2e^{-d})+2e^{-2d}(1+2d^2e^{-d})^2,
$$
where we used that $e^{-d(1-2de^{-d})}\ge
e^{-d}(1+2d^2e^{-d})$ which essentially is the inequality
$e^x>1+x$.

Combining the above we see that it suffices to establish
for $d\ge\pi$
the inequality
\begin{equation}\label{eqforF}
\aligned
F(d):=\left(1-2e^{-d}(1+2d^2e^{-d})+2e^{-2d}(1+2d^2e^{-d})^2\right)
\left((1-2de^{-d})^2+1\right)-\\-
2(1-2e^{-d})(1-2de^{-d})\le0.
\endaligned
\end{equation}
Simplifying $F(d)$ we obtain
$$
F(d)=4e^{-2d}F_1(d),
$$
where
$$
\aligned
F_1(d)=
-d^2+
2d^2e^{-d}+4d^3e^{-d}+1-2de^{-d}+2d^2e^{-2d}-
8d^3e^{-2d}+\\+8d^4e^{-3d}-8d^5e^{-3d}+8d^6e^{-4d}.
\endaligned
$$
Next, dropping all  negative terms  except for $-d^2$ we
have
$$
F_1(d)<
-d^2+1+F_2(d),\qquad F_2(d):=
2d^2e^{-d}+4d^3e^{-d}+2d^2e^{-2d}+8d^4e^{-3d}+8d^6e^{-4d}.
$$
Each term in $F_2(d)$ is monotonely decreasing for
$d\ge\pi$, and hence
$$
F_1(d)<-\pi^2+1+F_2(\pi)=-2.530\ldots<0,
$$
which proves~\eqref{forallD}.

To explain our choice of $\lambda_*(D)$ in~\eqref{lam*}
 we find the asymptotics as $D\to\infty$  of the inverse function
$\lambda=\lambda(D)$  since our   $\lambda_*(D)$ contains the first two terms
of this asymptotic expansion.
We have
$$
\aligned
&f(\lambda)=
\frac1{2\lambda^{1/2}}\tanh\frac{\lambda^{1/2}}2=
\frac1{2\lambda^{1/2}}\frac{1-e^{-\lambda^{1/2}}}{1+e^{-\lambda^{1/2}}}=
\frac12\lambda^{-1/2}\left(1-2e^{-\lambda^{1/2}}+2e^{-2\lambda^{1/2}}+
\dots\right),\\
&g(\lambda)=
\frac14\lambda^{-3/2}\left(1+(-2-2\lambda^{1/2})e^{-\lambda^{1/2}}+
(2+4\lambda^{1/2})e^{-2\lambda^{1/2}}+\dots\right),\\
&h(\lambda)=
\frac14\lambda^{-1/2}\left(1+(-2+2\lambda^{1/2})e^{-\lambda^{1/2}}+
(2-4\lambda^{1/2})e^{-2\lambda^{1/2}}+\dots\right),
\endaligned
$$
and
$$
\aligned
D(\lambda)=\frac{h(\lambda)}{g(\lambda)}=
\frac{\lambda(1+(-2+2\lambda^{1/2})e^{-\lambda^{1/2}}+
(2-4\lambda^{1/2})e^{-2\lambda^{1/2}}+\dots)}
{1+(-2-2\lambda^{1/2})e^{-\lambda^{1/2}}+(2+4\lambda^{1/2})
e^{-2\lambda^{1/2}}+\dots}=\\
=
\lambda+4\lambda^{3/2}e^{-\lambda^{1/2}}
+8\lambda^{2}e^{-2\lambda^{1/2}}+\dots.
\endaligned
$$
Hence the inverse function $\lambda(D)$ has the asymptotics
$$
\lambda(D)=D(1-4\sqrt{D}e^{-\sqrt{D}}+\dots)=
d^2(1-2de^{-d}+\dots)^2,
$$
whose first two terms give~\eqref{lam*}.

Finally,
$$
\aligned
\mathbb{V}(D)=
\frac1{2\sqrt{\lambda(D)}}\tanh\frac{\lambda(D)^{1/2}}2\cdot(\lambda(D)+D)=\\
\frac12\frac1{d(1-2de^{-d}+\dots)}\tanh\frac{d(1-2de^{-d}+\dots)}2
\left(d^2(1-2de^{-d}+\dots)^2+d^2\right)=\\
d\cdot\tanh\frac{d(1-2de^{-d}+\dots)}2
\left(1-2de^{-d}+\dots\right)\left(1+2de^{-d}+\dots\right)=\\
d\cdot\tanh\frac{d(1-2de^{-d}+\dots)}2(1+\dots)=
d(1-2e^{-d}+\dots),
\endaligned
$$
which proves sharpness and completes the proof.
\end{proof}
\begin{remark}\label{R:3terms}
{\rm
Arguing  similarly to
Proposition~\ref{Prop6.as} one can write down the
three-term expansion of $\mathbb{V}(D)$ as $D\to \infty$:
\begin{equation}\label{Dir.max-as}
\mathbb V(D)=
D^{1/2}-2D^{1/2}e^{-D^{-1/2}}-2D^{1/2}(D-1)e^{-2D^{-1/2}}
+o(e^{-3D^{-1/2}+\varepsilon}).
 \end{equation}
 The third term is negative and hence
 inequality~\eqref{forallD} holds for all sufficiently large
 $D$. However, as we have shown, \eqref{forallD}
 holds for {\it all} $D\ge\pi^2$.
}
\end{remark}

\par
\subsubsection{On a second order inequality on the interval}
We want to apply the above developed theory to the following
one-dimensional interpolation inequality
\begin{equation}\label{8.simple}
\|u\|^2_{\infty}\le K\|u\|_{L_2(0,1)}^{3/2}\|u''\|_{L_2(0,1)}^{1/2},
\ \ u\in H^2(0,1)\cap H^1_0(0,1)
\end{equation}
Here, $Au(x):=u^{(4)}(x)$ with  boundary conditions
$u(0)=u(1)=u''(0)=u''(1)=0$, $B=Id$, $\theta=3/4$ and
$\Bbb A(\lambda)=u^{(4)}+\lambda u$ and, in order to find the key function
$G_\lambda(\xi,\xi)$, we need to solve the equation
\begin{equation*}\label{u4}
u''''(x)+\lambda u(x)=\delta(x-\xi).
\end{equation*}
Using the orthonormal system of eigenfunctions
$\{\sqrt{2}\sin \pi n x\}_{n=1}^\infty$ we obtain
\begin{equation}\label{Green_u4}
G_\lambda(x,\xi)=2\sum_{n=1}^\infty\frac{\sin \pi nx\sin \pi n\xi}
{\pi^4n^4+\lambda},
\end{equation}
and setting $\lambda=4a^4\pi^4$ to simplify the formulas below we have
\begin{equation*}\label{Greenxixi}
G_\lambda(\xi,\xi)=\frac2{\pi^4}\sum_{n=1}^\infty\frac{\sin^2 \pi n\xi}
{n^4+4a^4}.
\end{equation*}
Next, the identity $\sin^2\pi n\xi=1/2-
({e^{2\pi i n \xi}+e^{-2\pi i n \xi}})/4$
and the Poisson summation formula
$$
\sum_{n=-\infty}^\infty g(n+y)=
\sqrt{2\pi}\sum_{n=-\infty}^\infty e^{2\pi i ny}
\widehat{g}(2\pi n)
$$
give
\begin{equation*}\label{Greenxixiexp}
\aligned
\sum_{n=1}^\infty\frac{\sin^2 \pi n\xi}
{n^4+4a^4}=
\frac14\left(\sum_{n=-\infty}^\infty\frac1{n^4+4a^4}-
\sum_{n=-\infty}^\infty\frac{e^{2\pi i n
\xi}}{n^4+4a^4}\right)=\\=
\frac14\frac1{(\sqrt{2}a)^3}
\left(\sum_{n=-\infty}^\infty f(\sqrt{2}\,a 2\pi n)-
\sum_{n=-\infty}^\infty f(\sqrt{2}\,a 2\pi
(n+\xi))\right)=\\=
\frac14\frac1{(\sqrt{2}a)^3}\frac{\pi\sqrt{2}}2
\left(\sum_{n=-\infty}^\infty f_0(a 2\pi n)-
\sum_{n=-\infty}^\infty f_0(a 2\pi
(n+\xi))\right),
\endaligned
\end{equation*}
where
$$
f(y):=\int_{-\infty}^\infty\frac{e^{-ixy}\,dx}{x^4+1}=
\frac{\pi\sqrt{2}}2f_0\left(\tfrac{|y|}{\sqrt{2}}\right),
\qquad
f_0(x)=e^{-|x|}(\cos|x|+\sin|x|).
$$
Combining the above, we obtain for the key function in
Theorem~\ref{Th4.main1}
\begin{equation}\label{sum_sum}
\lambda^{3/4}G_\lambda(\xi,\xi)=\frac{\sqrt{2}}4
\left(\sum_{n=-\infty}^\infty f_0(a 2\pi n)-
\sum_{n=-\infty}^\infty f_0(a 2\pi
(n+\xi))\right)=:\frac{\sqrt{2}}4S(a,\xi).
\end{equation}
We now observe that the function $f_0(x)$ (which  up to
a constant factor is the fundamental solution of the operator
\eqref{u4} on the whole line) has positive local maximums at
$x=2\pi n$ (the one at $0$ being the global) and negative
local minimums at $x=\pi+2\pi n$
(the ones at $\pm\pi$ being the global).

For $a$ large enough and $\xi\le 1/2$ the leading terms are
$$
S(a,\xi)=f_0(0)-f_0(2\pi a\xi)+O(e^{-2\pi a}),
$$
and since $f_0(x)$ has a negative global minimum at $\pi$,
we see a ``horn'' of height $e^{-\pi}$ at $\xi=a^{-1}/2$,
which gives the maximum value of $S(a,\xi)$ for large $a$;
see Fig.~\ref{pic_2d_n_9_10}.

As for the global maximum of $S(a,\xi)$, we see that if
$a=1$ and $\xi=1/2$, then the first and the second sums in~\eqref{sum_sum}
count one by one  all the maximums  and all the negative minimums of
 $f_0$, respectively. Therefore
$$
S(a,\xi)\le S(1,1/2)=\coth\frac\pi 2.
$$
Therefore we have proved the following result.
\begin{theorem}\label{Th:4th_order}
The sharp constant in inequality~\eqref{8.simple}
is
$$
K=\frac{\sqrt{2}}{\sqrt[4]{27}}\cdot \coth\frac\pi 2=
\frac{\sqrt{2}}{\sqrt[4]{27}}\cdot 1.09033\dots\,.
$$
The unique extremal function is given by~\eqref{Green_u4} with
$\lambda=4\pi^4$ and $\xi=\frac12$.
\end{theorem}
\begin{remark}\label{R:0and00}
{\rm
We point out that the sharp constant
in~\eqref{8.simple} for $u\in H^2_0(0,1)$ is the same as
that on the whole line, namely,
$\frac{\sqrt{2}}{\sqrt[4]{27}}$.
}
\end{remark}
\begin{figure}[ht]
\centering
\begin{tabular}{cc}
\includegraphics*[angle=0,width=2.5in]{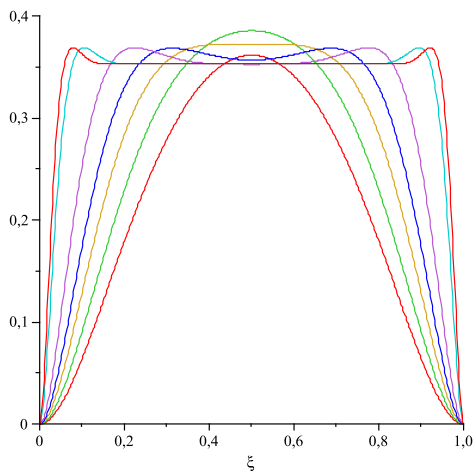}
\end{tabular}
\caption{Plots of $\xi\to S(a,\xi)$ for $a=2.5/\pi$ (red),
$a=1$ (green), $a=4/\pi$ (brown), $a=5/\pi$ (blue), $a=(7,15,20)/\pi$}
\label{pic_2d_n_9_10}
\end{figure}

\subsubsection{A second-order inequality in 3D}
In conclusion we consider a second-order  inequality in
a three dimensional domain $\Omega\subseteq\mathbb{R}^3$
for which the passage from $H^2_0(\Omega)$ to
a wider space $H^2(\Omega)\cap H^1_0(\Omega)$
does not increase the constant in the corresponding
interpolation inequality. This inequality was obtained
in~\cite{Xie} by a somewhat problem specific method, nevertheless
we present its proof in the
framework of our general Theorem~\ref{Th4.main1}.
\begin{theorem}\label{Th:ketai}
Let $\Omega\subseteq\mathbb{R}^3$ be an arbitrary
domain. Let $\dot{H}^1_0(\Omega)$ be the completion
of $C^\infty_0(\Omega)$ in the norm
$\|\nabla\cdot\|_{L_2(\Omega)}$. Then for
$u\in \dot{H}^1_0(\Omega)\cap\{u:\ \Delta u\in L_2(\Omega)\}$
the following inequality holds
\begin{equation}\label{Xie}
\|u\|_\infty^2\le\frac1{2\pi}\|\nabla u\|\|\Delta u\|,
\end{equation}
where the constant is sharp and no extremal functions
exist,
unless $\Omega=\mathbb{R}^3$ and
$$
u(x)=\mathcal{F}^{-1}(1/(|\xi|^2+|\xi|^4)).
$$
\end{theorem}
\begin{proof}
We first assume   that $\Omega$ is a bounded domain with
smooth boundary. Then by the elliptic regularity we have
$$
u\in \dot{H}^1_0(\Omega)\cap\{u:\ \Delta u\in
L_2(\Omega)\}=
H^2(\Omega)\cap H^1_0(\Omega)
$$

In accordance with Theorem~\ref{Th4.main1} we have to
consider the Green's function $G_\lambda(x,\xi)$ of the
following 4th order elliptic equation
\begin{equation}\label{4th_ellipt}
\aligned
&(\Delta^2-\lambda\Delta)G_\lambda(x,\xi)=
-\Delta(-\Delta+\lambda)G_\lambda(x,\xi)=\delta(x-\xi),\\
&G_\lambda(x,\xi)=\Delta G_\lambda(x,\xi)=0, \
\xi\in\Omega,\ x\in\partial\Omega.
\endaligned
\end{equation}

We denote by $g_D(x,\xi)$ the Green's function of the
Dirichlet Laplacian:
$$
-\Delta g_D(x,\xi)=\delta(x-\xi),\qquad
g_D(x,\xi)=0, \
\xi\in\Omega,\ x\in\partial\Omega,
$$
and by
$g_H(x,\xi)$ the Green's function of  the
Helmholtz equation:
$$
(-\Delta+\lambda) g_H(x,\xi)=\delta(x-\xi),\qquad
g_H(x,\xi)=0, \
\xi\in\Omega,\ x\in\partial\Omega.
$$
By the maximum principle we have
\begin{equation}\label{max_prin}
0<g_D(x,\xi)<\frac1{4\pi|x-\xi|},
\qquad0<g_H(x,\xi)<\frac{e^{-\sqrt{\lambda}|x-\xi|}}{4\pi|x-\xi|},
\end{equation}
where the functions on the right-hand sides are the
corresponding fundamental solutions in $\mathbb{R}^3$.
Therefore
$$
\aligned
0&<G_\lambda(x,\xi)=\int_{\Omega}
g_H(x,y)g_D(y,\xi)dy<\int_{\Omega}
\frac{e^{-\sqrt{\lambda}|x-y|}}{4\pi|x-y|}\,\frac1{4\pi|y-\xi|}dy<\\&<
\int_{\mathbb{R}^3}
\frac{e^{-\sqrt{\lambda}|x-y|}}{4\pi|x-y|}\,\frac1{4\pi|y-\xi|}dy=
\frac1{(2\pi)^3}\int_{\mathbb{R}^3}\frac{e^{i(\xi-x)\cdot z}dz}
{|z|^2(|z|^2+\lambda)}\,,
\endaligned
$$
the function on the right-hand side being the fundamental
solution of~\eqref{4th_ellipt} in $\mathbb{R}^3$, and, consequently,
$$
G_\lambda(\xi,\xi)<
\frac1{(2\pi)^3}\int_{\mathbb{R}^3}\frac{dz}
{|z|^2(|z|^2+\lambda)}=\frac1{4\pi\sqrt{\lambda}}\,.
$$
Now \eqref{Xie} follows from Theorem~\ref{Th4.main1}.
The constant $1/(2\pi)$ is clearly sharp, since we have the inequality
with the same constant for $u\in H^2_0(\mathbb{R}^3)$, see~\eqref{cR^n},
and (using extension by zero) for
$u\in H^2_0(\Omega)\subset
\dot{H}^1_0(\Omega)\cap\{u:\ \Delta u\in L_2(\Omega)\}$ .

Exhausting an arbitrary $\Omega\subset\mathbb{R}^3$ with
bounded domains $\Omega_n$ one can show that \eqref{Xie}
holds in the general case (see \cite{Xie} for the details).
\end{proof}

\begin{remark}\label{R:R30and00}
{\rm Unlike the previous example, the fundamental solution
of~\eqref{4th_ellipt} in $\mathbb{R}^3$ is positive. }
\end{remark}

\noindent{\bf Acknowledgments.}
The authors would like to thank A.A.Laptev and S.I.Pokhozhaev for  many helpful discussions.

This work was supported by the Russian Ministry of Education
and Science (contract no. 8502). The work of
A.A.I. was supported in
part by the Russian Foundation for Fundamental Research,
grants~no.~12-01-00203 and no.~11-01-00339.

\bibliographystyle{amsplain}

\enddocument